\documentclass[12pt]{amsart}

\usepackage[utf8]{inputenc}
\usepackage[english]{babel}
\usepackage{amsmath}
\usepackage{amsthm}
\usepackage{amssymb}
\usepackage{amsfonts}
\usepackage{graphicx}
\usepackage{hyperref}
\usepackage{cleveref}

\usepackage[foot]{amsaddr}

\theoremstyle{plain}
\newtheorem{theorem}{Theorem}[section]
\newtheorem{lemma}[theorem]{Lemma}
\newtheorem{proposition}[theorem]{Proposition}

\theoremstyle{definition}

\newtheorem{remark}[theorem]{Remark}
\newtheorem{assumption}[theorem]{Assumption}

\usepackage{enumitem}

\usepackage{tikz}
\usetikzlibrary{shapes.misc}
\usepackage{float}

\usepackage{algorithm}%
\usepackage{algorithmic}%

\newcommand{\minimize}[2]{\ensuremath{\underset
		{\substack{{#1}}}%
		{\text{\rm minimize}}\;\;#2 }}

\newcommand{\OAP}{ \mathcal{O}}

\usepackage{enumitem}

\title[Global solutions for the sensors placement problem]
{Global solutions for the sensors placement problem via weakly convex optimization}

\author{Giovanni Bruccola${}^{a,b}$}

\address{${}^a$Systems Research Institute, Polish Academy of Sciences, Newelska 6, 01-447 Warsaw, Poland}

\address{${}^b$Space Research Centre, Polish Academy of Sciences, Bartycka 18A, 00-716 Warszawa, Poland}

\begin{document}

\maketitle

\begin{abstract}
We address the problem of optimally placing a limited number of sensors to reconstruct high-dimensional signals without knowledge of the underlying dynamics. The task is formulated as a nonconvex combinatorial optimisation problem and recast as a weakly convex constrained projection problem. This reformulation allows us to compute $\varepsilon$-global solutions using the Inexact Cutting Sphere algorithm.
We further propose the Inverse Cutting Sphere algorithm, which starts from any feasible heuristic solution and either improves it by a prescribed tolerance $\varepsilon$ or certifies its $\varepsilon$-global optimality. 
The framework is evaluated on pressure reconstruction for NACA airfoils using XFOIL data. 
\end{abstract}

\begingroup

\par
\endgroup

		
\section{Introduction}

Given a test dataset $\mathcal{S}$ composed of signals $\omega_s \in \mathbb{R}^m$, $s = 1, \dots, S$, our goal is to determine the optimal positions of $p \ll m$ sensors and to reconstruct the full signals $\omega_s \in \mathbb{R}^m$ from the corresponding sensor outputs $y_s \in \mathbb{R}^p$.

We do not know the underlying physical laws that govern the behaviour of these signals. In other words, we cannot formulate a linear time-dependent dynamical system capable of modelling their evolution. However, we have access to a training dataset $\mathcal{T}$ consisting of $t$ snapshots. Consequently, $\mathcal{T}$ is a matrix in $\mathbb{R}^{t \times m}$.

Given a subset of $p$ sensor positions $I_p \in \mathcal{K}$ (where $\mathcal{K}$ denotes the collection of all possible $p$-element subsets of the $m$ candidate locations), the measurements obtained from the installed sensors for a signal $\omega_s$ ($s = 1, \dots, S$) are modelled as
\[
y_s = c_s^\top x + v_s, \quad i \in I_p,
\]
where $c_s \in \mathbb{R}^{p \times m}$, $x \in \mathbb{R}^m$, and $v_s \in \mathbb{R}^p$ represents measurement noise.

The main contributions of the paper are as follows:

\begin{enumerate}
    \item We model the sensor placement problem as weakly convex optimisation problems by reformulating the most discussed approaches in the literature, in particular the ones proposed in \cite{joshi2008sensor, singh2005determining}.
    \item To the best of our knowledge, we are the first to propose global optimisation algorithms for the sensor placement problem by means of the Inexact Cutting Sphere algorithm \cite{bednarczuk2026outer}, which finds a global $\varepsilon$-solution, and the Inverse Cutting Sphere algorithm, that can improve known feasible solutions (e.g., those obtained by heuristic methods).
    \item We compare the proposed global optimization algorithms with established heuristic approaches from the literature, in particular \cite{manohar2017data}.
\end{enumerate}

Sensor placement has been studied extensively for several decades (see, for example, the survey \cite{padula1999optimization}). When the underlying process can be described by a linear time-dependent dynamical system, the classical approach consists in placing sensors so as to maximise a suitable measure of observability. A system is said to be observable if the current state can be reconstructed from the sensor measurements alone for every possible evolution of the state and control inputs \cite{kalman1970lectures}. Observability is therefore a binary (on/off) property. Nevertheless, various quantitative measures of observability and associated optimisation techniques have been proposed in the literature (see, e.g., \cite{singh2005determining,polyak2013lmi,zhang2016optimal,zhang2017sensor,gagliardi2023traffic,gagliardi2024joint}).

When the system is nonlinear or its governing equations are unknown, but simulation data are available, one can construct an observability Gramian and minimise an appropriate observability measure (see, e.g., \cite{singh2005determining,krener2009measures,hinson2014observability,cellini2023empirical}). Constructing the Gramian requires the ability to simulate the signal $\omega(t)$ under perturbations of the initial conditions, as described in the aforementioned references.

In the present setting we are in a situation analogous to that considered in \cite{manohar2017data}, where the available snapshot dataset is not necessarily generated by perturbing initial conditions. The QDEIM method, originally introduced in \cite{drmac2016new} and further analysed in \cite{manohar2017data}, proceeds in two steps:
\begin{enumerate}
    \item A tailored basis is constructed, for instance via proper orthogonal decomposition (POD). This amounts to computing the singular value decomposition (SVD) of the snapshot matrix $\mathcal{T}$.
    \item All $m$ candidate sensor locations are ranked by means of a QR decomposition with column pivoting, and the first $p$ positions are retained.
\end{enumerate}

In \cite{manohar2017data}, the reconstruction quality obtained with sensors selected by QDEIM is compared with that achieved by compressed-sensing techniques (see, e.g., \cite{candes2006near,candes2006compressive}).

In \cite{Joshi2009}, the sensor placement problem is formulated as a non-convex optimisation problem (detailed in Subsection~\ref{sec: models}). The authors propose solving a convex relaxation of the original non-convex formulation. Our approach, by contrast, consists of the following three steps:
\begin{enumerate}
    \item A tailored basis is obtained via the POD method, exactly as in \cite{manohar2017data}.
    \item The non-convex sensor placement problem is formulated following the approach of \cite{Joshi2009}.
    \item A globally optimal (or $\varepsilon$-globally optimal) solution of the non-convex problem is computed by recasting it as a weakly convex optimisation problem and applying \cite[Algorithm 3]{bednarczuk2026outer}.
\end{enumerate}

While this work show how to modify the DEIM method by introducing global optimization algorithms, different papers propose to modify the DEIM method with Bayesian estimation, see \textit{e.g.} \cite{barklage2024sensor, jiang2026optimization}. The result of this modification is to move the sensor placement problem from the deterministic realm into the non-deterministic one. The deterministic and non-deterministic approaches are compared in \cite{kakasenko2026bridging}. 
In the present paper, however, we focus on the deterministic case. Combining Bayesian estimation and our global optimization approach may be the topic of future investigations.

The remainder of the paper is organised as follows. In Section~\ref{sec: not} we introduce the necessary mathematical background. Section~\ref{sec: full procedure} describes the complete procedure proposed for reconstructing the signals $\omega_s$, $s = 1, \dots, S$. In Section~\ref{sec: models} we review the most common optimisation models for sensor placement that appear in the literature. Section~\ref{sec: reform} reformulates these models as weakly convex constrained projection problems (see Problem~\ref{prob: proj}); this is a key step towards obtaining $\varepsilon$-global solutions via the Inexact Cutting Sphere algorithm \cite[Algorithm 3]{bednarczuk2026outer}. In Section~\ref{sec: inv} we present the Inverse Cutting Sphere algorithm, which can improve a given heuristic solution or certify its $\varepsilon$-global optimality. Finally, Section~\ref{sec: exp} reports numerical experiments.

		\section{Notation, preliminaries and subdifferentials}
\label{sec: not}

We denote by $S^n$, $S_+^n$ and $S_{++}^n$ the space of symmetric $n\times n$ matrices and the cones of symmetric positive semidefinite and positive definite matrices, respectively. We denote by $\operatorname{trace}(\cdot,\cdot)$ the inner product on $S_+^n$ defined by
\[
\operatorname{trace}(A,B) = \sum_{i=1}^n \sum_{j=1}^n A_{ij} B_{ij},
\]
where $A,B\in S_+^n$ and $A_{ij}$, $B_{ij}$ are the entries of $A$ and $B$ in position $(i,j)$. In general, we write $\langle \cdot,\cdot\rangle$ for the inner product on a Hilbert space $H$ and $\|\cdot\|$ for the associated norm (Euclidean norm when $H=\mathbb{R}^n$). We use \emph{lsc} as an abbreviation for lower semi-continuous.

Given a Hilbert space $H$, the convex (Moreau) subdifferential of a proper function $f:H\to (-\infty,+\infty]$ at a point $x\in H$ is defined as \cite[Definition 16.1]{bauschke2017convex}
\begin{equation}
\partial f(x) := \bigl\{ v\in H \ \big|\  f(y)-f(x) \ge \langle v,y-x\rangle \ \ \forall\, y\in H \bigr\}.
\end{equation}

Let $\lambda_{\max}(\cdot)$ and $\lambda_{\min}(\cdot)$ are the functions returning the largest and smallest eigenvalues of a matrix $A\in S^n$, respectively. If $v$ is a leading eigenvector of $A$, i.e., $\lambda_{\max}(A)v=Av$, then for every $B\in S^n$ we have
\[
\operatorname{trace}(vv^T,B-A) = v^TBv - \lambda_{\max}(A) \le \lambda_{\max}(B)-\lambda_{\max}(A).
\]
Consequently, $vv^T\in\partial\lambda_{\max}(A)$. Similarly, if $u$ is an eigenvector of $A$ corresponding to its smallest eigenvalue, then $-uu^T\in\partial(-\lambda_{\min})(A)$.

We define the condition number of a matrix $A\in S_{++}^n$ by
\[
\kappa(A) := \frac{\lambda_{\max}(A)}{\lambda_{\min}(A)}.
\]

Following \cite[equation 1.1.7 (sup)]{pallaschke2013foundations}, a function $f$ is said to be \emph{abstract convex} (or $\Phi$-convex) with respect to a class of elementary functions $\Phi$ if, for every $x\in H$,
\[
f(x) = \sup\bigl\{\varphi(x)\ \big|\ \varphi\in\Phi\bigr\}.
\]
A function $\varphi\in\Phi$ is said to belong to the $\Phi$-subdifferential of $f$ at $x\in H$ if
\begin{equation}
f(y)-f(x) \ge \varphi(y)-\varphi(x) \qquad \forall\, y\in H.
\end{equation}

The class $\Phi_{\mathrm{lsc}}$ of elementary functions is defined by
\begin{equation}
\label{def: class psi}
\Phi_{\mathrm{lsc}} := \bigl\{ \varphi:H\to\mathbb{R} \ \big|\ 
\varphi(u) = -a\|u\|^2 + \langle b,u\rangle + c,\ 
b\in H,\ a\in\mathbb{R}_{++},\ c\in\mathbb{R} \bigr\}.
\end{equation}

A function $f$ is called \emph{$\rho$-weakly convex} (for $\rho\ge 0$) if $f + \rho\|\cdot\|^2$ is convex. Every $\rho$-weakly convex function is $\Phi_{\mathrm{lsc}}$-convex. Indeed, if $\tilde{f}=f + a\|\cdot\|^2$ with $a\ge\rho$ is convex and $b\in\partial\tilde{f}(x)$, then the function
\[
\varphi_{(a,b)}(y) := -a\|y\|^2 + \langle b,y\rangle
\]
belongs to the $\Phi_{\mathrm{lsc}}$-subdifferential of $f$ at $x$, i.e., $(a,b)\in\partial_{\mathrm{lsc}}f(x)$.
		
		Finally, given a point $\overline{x}$, a constraints set $\mathcal{A}$ and an objective function $f$, we say that $\overline{x}$ is a global $\varepsilon$-solution, $\varepsilon>0$, if the following holds.
		\begin{equation}
		    \label{eps_glob}
		    (\overline{x}\in \mathcal{A}), \ \ f(\overline{x}) \le \min_{x\in \mathcal{A}} f(x)+\varepsilon.
		\end{equation}
		
	\section{Sensors placement procedure for reconstruction}
\label{sec: full procedure}

Inspired by \cite{manohar2017data}, we propose a \emph{training procedure} that takes the matrix $\mathcal{T}$ as input and returns the optimal positions of $p$ sensors together with two matrices $A_p \in \mathbb{R}^{p \times p}$ and $A \in \mathbb{R}^{p \times m}$, which are described below. Given the sensor outputs $y_s \in \mathbb{R}^p$ for $s \in \{1, \dots, S\}$, the reconstructed signals are then obtained as $\overline{\omega}_s = A c$, where $c \in \mathbb{R}^p$ is the least-squares solution of the linear system
\begin{equation}
\label{eq: least square}
A_p c = y_s.
\end{equation}

When the sensor outputs are corrupted by noise, the signals $\omega_s$ can be reconstructed following the approach described in \cite{Karnik2026}.

Our training procedure consists of two steps:
\begin{enumerate}
    \item Compute the matrix $A$ via \emph{proper orthogonal decomposition} (POD) applied to $\mathcal{T}$, as described in \cite{manohar2017data}. 
    In particular, we find $A$ by taking the first $p$ linearly independent eigenvectors of one of the matrices found by the SVD decomposition, described in \cite[equation 13]{manohar2017data}.
    \item Determine the optimal positions of the $p$ sensors among all $m$ candidate locations by solving one of the nonconvex sensor selection problems introduced in the next section to $\varepsilon$-global optimality. Equivalently, this amounts to selecting the $p$ columns of $A \in \mathbb{R}^{p \times m}$ that form the matrix $A_p \in \mathbb{R}^{p \times p}$.
\end{enumerate}

To compute an inexact global solution to the sensor selection problems, we employ the Inexact Cutting Sphere algorithm \cite[Algorithm 3]{bednarczuk2026outer}. To the best of our knowledge, this is the first time that an inexact global solution for sensor selection problems is obtained in this manner.
Then, we also propose a new algorithm, the Inverse Cutting Sphere algorithm (Algorithm \ref{alg: inv inexact}) to improve a known feasible solution.
The remainder of the literature relies instead on convex relaxations or heuristic methods.	
		
	\section{Sensors placement as an optimization problem}
\label{sec: models}

We begin this section with the sensor placement problem introduced in \cite{Joshi2009}, which differs slightly from the formulation presented in the introduction. Our objective is to estimate a vector $x \in \mathbb{R}^n$ from $m$ possible linear measurements of the form
\begin{equation}
y_i = a_i^\top x + v_i, \qquad i = 1, \dots, m,
\end{equation}
where the noise terms $v_1, \dots, v_m$ are independent random variables distributed as $\mathcal{N}(0,1)$.

Interpreting $m$ as the number of candidate sensor locations, we aim to select an optimal subset of $p \ge n$ sensors. To this end, we introduce binary decision variables $z_i \in \{0,1\}$, $i = 1, \dots, m$, where $z_i = 1$ if and only if location $i$ is selected.

It was shown in \cite{Joshi2009} that the following optimization problem minimizes the volume of the confidence ellipsoid associated with the estimation error:
\begin{equation}
\tag{SL1}
\label{prob: sensor placement 1}
\begin{aligned}
\max_z \quad & \log \det \Bigl( \sum_{i=1}^m z_i a_i a_i^\top + \delta I \Bigr) \\
\text{s.t.} \quad & \mathbf{1}^\top z = p, \\
& z_i \in \{0,1\}, \quad i = 1, \dots, m.
\end{aligned}
\end{equation}
The regularization term $\delta I$ with $\delta > 0$ was added by us. Note that, once a set of indices $I \subset \{1, \dots, m\}$ with $|I| = p$ has been selected, the measurement model takes the form $y = Ax + v$, where the rows of the matrix $A \in \mathbb{R}^{p \times n}$ are the vectors $a_i$ for $i \in I$, and the components of the noise vector $v \in \mathbb{R}^p$ are the corresponding $v_i$.

An alternative formulation, also proposed in \cite{Joshi2009}, is given by
\begin{equation}
\tag{SL2}
\label{prob: sensor placement 2}
\begin{aligned}
\min_z \quad & \operatorname{trace} \Bigl( \sum_{i=1}^m z_i a_i a_i^\top + \delta I \Bigr)^{-1} \\
\text{s.t.} \quad & \mathbf{1}^\top z = p, \\
& z_i \in \{0,1\}, \quad i = 1, \dots, m.
\end{aligned}
\end{equation}

Following the approach of \cite{singh2005determining,krener2009measures,hinson2014observability}, we may replace the $\log\det$ objective in \eqref{prob: sensor placement 1} by the matrix condition number function $\kappa : S_{++}^n \to \mathbb{R}$, defined by
\[
\kappa(A) := \frac{\lambda_{\max}(A)}{\lambda_{\min}(A)},
\]
where $\lambda_{\max}(\cdot)$ and $\lambda_{\min}(\cdot)$ denote the largest and smallest eigenvalues, respectively. This yields the problem
\begin{equation}
\tag{SL3}
\label{prob: sensor placement 3}
\begin{aligned}
\min_z \quad & \kappa\Bigl( \sum_{i=1}^m z_i a_i a_i^\top + \delta I \Bigr) \\
\text{s.t.} \quad & \mathbf{1}^\top z = p, \\
& z_i \in \{0,1\}, \quad i = 1, \dots, m.
\end{aligned}
\end{equation}

Problem \eqref{prob: sensor placement 3} differs substantially from \eqref{prob: sensor placement 1} and \eqref{prob: sensor placement 2}, since the condition number $\kappa$ is a nonconvex function.

\section{Global solutions algorithm for the sensors placement problems}
\label{sec: reform}

\subsection{Motivation for the cutting sphere method}
\label{chap: reformulation}

The focus of this section is to describe an algorithm that solves problems \eqref{prob: sensor placement 1} and \eqref{prob: sensor placement 2} to global optimality. To this end, we reformulate these two problems as weakly convex constrained problems and apply the cutting sphere algorithm with warm restart \cite[Algorithm 3]{bednarczuk2026outer}.

Algorithm \cite[Algorithm 3]{bednarczuk2026outer} solves,  to $\varepsilon$-global optimality, nonconvex projection problems of the form
\begin{equation}
\label{prob: proj}
\tag{NPP}
\begin{split}
&\operatorname{min}_{x\in P}\, \|x\|^2\\
&\text{s.t. }x\in \mathcal{A},\\
&\text{where }\mathcal{A}:=\{x\in P\ |\ f_i(x)\le 0,\ i=1,...,m\},
\end{split}
\end{equation}
    where $f_i:\mathbb{R}^n\rightarrow (-\infty, +\infty]$, $i=1,...,m$, are proper, lower semicontinuous and $\rho_i$-weakly convex functions. The set $P$ is a polyhedron (it was simply $\mathbb{R}^n$ in \cite[Equation 1]{bednarczuk2026outer}). Moreover, $\mathcal{A}\neq \varnothing$ and there exists $a\ge 0$ with ${a>\operatorname{Opt}(\eqref{prob: proj})}$ such that the functions $f_i$, $i=1,...,m$, are continuous on an open set containing the level set $\operatorname{lev}_{\le a}\|x\|^2$.
 
For the sensor placement problems, we have, for some $m\le n$,    \begin{equation}
\label{eq:poly}
P := \bigl\{ x=(x^{n-m}, x^{m}) \in \mathbb{R}^n \ \big|\ 
(\mathbf{1},0)^T x^{m} = p;\quad 0 \le x_i \le 1,\ i=1,...,m \bigr\}.
\end{equation}

We will therefore obtain global $\varepsilon$-solutions for the reformulations of problems \eqref{prob: sensor placement 1}, \eqref{prob: sensor placement 2} and \eqref{prob: sensor placement 3}. To the best of our knowledge, this constitutes a new result in the literature on sensor placement problems.

Reformulating problems \eqref{prob: sensor placement 1} and \eqref{prob: sensor placement 2} in the form of \eqref{prob: proj} is possible thanks to the following theorem.

\begin{theorem}\cite[Theorem 1]{bednarczuk2026outer}
\label{th: minnoncon}
Assume that Problem
\begin{equation}
\label{prob: proj_pre}
\begin{split}
&\operatorname{min}_{z\in \mathbb{R}^{n-1}}\, F(z)\\
&z\in \overline{\mathcal{A}},
\end{split}
\end{equation}
has a solution, where $\overline{\mathcal{A}}$ is some closed set (for the sensor placement problems, $\overline{\mathcal{A}}$ is intersected with $P$ in \eqref{eq:poly}). Let $(\rho,\eta)\in\, ]0,+\infty[ \times \mathbb{R}$ be such that
\begin{equation}
\label{set: D}
D = \left\{\widehat{z}\in \arg\min_{z\in S}F(z) \mid F(\widehat{z}) + \eta \ge \frac{\rho}{2}\|\widehat{z}\|^2\right\}\neq\varnothing.
\end{equation}
For $y = (\overline{y}, y_{n}) \in \mathbb{R}^{n-1}\times \mathbb{R}$, let
\begin{equation}
\label{eq: new constr form}
\mathcal{A}:=
\left\{y\in \mathbb{R}^n \ |\ \begin{cases}
F(\overline{y}) + \eta - \frac{\rho}{2}\|y\|^2\le 0; \\
\overline{y}\in \overline{\mathcal{A}}
\end{cases}\right\}.
\end{equation}
Then $\widehat{x} \in D$ if and only if there exists $\widehat{y}_{n}\in \mathbb{R}$ such that $\widehat{y} = (\widehat{x}, \widehat{y}_{n})$ is a solution to
\begin{equation}
\label{problem2}
\minimize{y\in\mathbb{R}^{n}}\,\|y\|^2\quad \text{s.t.}\quad y\in \mathcal{A}.
\end{equation}
\end{theorem}

Theorem \ref{th: minnoncon} is formulated in a more general form than \cite[Theorem 1]{bednarczuk2026outer}. Its proof can be obtained by repeating all the steps of the proof of \cite[Theorem 1]{bednarczuk2026outer} and is therefore omitted.

Note that \cite[Algorithm 3]{bednarczuk2026outer} is computationally demanding. Our aim is therefore to apply it to small instances of problems \eqref{prob: sensor placement 1}, \eqref{prob: sensor placement 2} and \eqref{prob: sensor placement 3} and to compare the results with the most established methods in the literature, in particular those of \cite{Joshi2009, manohar2017data}. These methods are heuristics and provide no optimality guarantees. The comparison will allow us to assess how close the solutions they produce are to the global solution, at least in the context of the experiments presented below.

\subsection{Problem SL1 reformulation for the cutting sphere}

In order to apply \cite[Algorithm 3]{bednarczuk2026outer} to Problem \eqref{prob: sensor placement 1}, we reformulate it as a weakly convex constrained problem whose objective is the squared Euclidean norm.

In the following, let $z\in \mathbb{R}^m$ (we will see later why this is not a restrictive assumption compared with $z\in \{0,1\}^m$). The objective function $\max_z \log\det\bigl(\sum_{i=1}^m z_i a_i a_i^T + \delta I\bigr)$ can be rewritten as
\[
f(z) := \min_{z\in \mathbb{R}^n} -\log\det\Bigl( \sum_{i=1}^m z_i a_i a_i^T + \delta I \Bigr).
\]
Let $A_i = a_i a_i^T$ for $i=1,...,m$ and define $A(z) = \sum_{i=1}^m A_i z_i + \delta I$. In our setting, the vectors $a_i\in \mathbb{R}^n$, $i=1,...,m$ are assumed to be linearly independent, and we select exactly $p = n$ of the variables $z_i$ to be equal to one. Note that $\det\bigl(\sum_{i=1}^m z_i a_i a_i^T + \delta I\bigr) > 0$.

It is well known that $-\log\det(\cdot)$ is convex on the cone of symmetric positive definite matrices. Consequently, $-\log\det(A(z))$ is convex, and
\[
-\nabla_z\log\det\Bigl( \sum_{i=1}^m z_i A_i \Bigr) = -\operatorname{tr}(A_i A^{-1}(z)), \quad i=1,...,m,
\]
see \cite[Example A.3]{Joshi2009}. We now state the following result.

\begin{proposition}
 Problem \eqref{prob: sensor placement 1} can be rewritten in the variable $x:=(z_1,...,z_m, x_{m+1})\in \mathbb{R}^{m+1}$ as
\begin{equation}
\label{prob: SL final}
\tag{SL1.2}
\begin{split}
& \min_{x \in \mathbb{R}^{n+1}} \|x\|^2 \\
& f_\Omega(x)\le 0;\quad (\mathbf{1},0)^T x = p;\quad g(x)\le 0, \\
& 0\le x_i \le 1,\quad i=1,...,n.
\end{split}
\end{equation}
where $\Omega$ is the optimal value of the auxiliary problem \eqref{prob: aux} below, $\eta := \Omega - p^2$,
\[
f_\Omega(x) := -\log\det\Bigl( \sum_{i=1}^m z_i A_i + \delta I \Bigr) - \eta - \|x\|^2,
\]
and
\[
g(x) = g(z) := \sum_{i=1}^m \Bigl| \Bigl(z_i - \tfrac12\Bigr)^2 - \tfrac14 \Bigr| \le 0.
\]
\end{proposition}

\begin{proof}
The binary constraints on the variables $z_i$, $i=1,...,m$ can be equivalently written as
\[
g(x) = g(z) := \sum_{i=1}^m \Bigl| \Bigl(z_i - \tfrac12\Bigr)^2 - \tfrac14 \Bigr| \le 0.
\]
The function $g$ is proper, lower semicontinuous and weakly convex on $\mathbb{R}^n$ (see \cite{bednarczuk2023forward}).

Consider the auxiliary convex optimization problem
\begin{equation}
\label{prob: aux}
\begin{split}
& \min_z f(z) \\
& \mathbf{1}^T z = p, \quad 0\le z_i\le 1,\quad i=1,...,m.
\end{split}
\end{equation}
Problem \eqref{prob: aux} is a continuous relaxation of \eqref{prob: sensor placement 1}, so its optimal value $\Omega$ can be computed efficiently. Let $z^*$ be a global minimizer of \eqref{prob: sensor placement 1}. Then
\begin{equation}
\label{eq: chain Pesquet}
f(z^*) \ge \Omega = \Omega + \|z^*\|^2 - \|z^*\|^2 \ge \Omega + \|z^*\|^2 - p^2 = \eta + \|z^*\|^2.
\end{equation}
The first inequality holds because \eqref{prob: aux} is a relaxation of \eqref{prob: sensor placement 1}. The second inequality follows from $z^*\in\{0,1\}^m$ and $\mathbf{1}^T z = p\ge 1$.
Therefore, we have $-\|z^*\|^2=-p\ge -p^2$ (we have an alternative reformulation if we take $p$ instead of $p^2$).
The last equality follows from the definition $\eta = \Omega - p^2$.

Since \eqref{eq: chain Pesquet} holds, we may apply \cite[Theorem 1]{bednarczuk2026outer}. Introducing the extended variable $x = (z_1,...,z_m, Z) \in \mathbb{R}^{n+1}$, we can rewrite problem \eqref{prob: sensor placement 1} as
\begin{equation}
\label{prob: aux 2}
\begin{split}
& \min_x \|x\|^2 \\
& f_\Omega(x)\le 0, \\
& \mathbf{1}^T z = p, \quad z_i\in\{0,1\},\quad i=1,...,m.
\end{split}
\end{equation}
Finally, \eqref{prob: sensor placement 1} is equivalent to \eqref{prob: SL final}.
\end{proof}

\subsection{Problem SL3 reformulation for the cutting sphere}

Reformulating problem \eqref{prob: sensor placement 3} is more involved because the condition number $\kappa$ is nonconvex. To apply \cite[Theorem 1]{bednarczuk2026outer}, we first consider the following equivalent reformulation.

\begin{lemma}
 Problem \eqref{prob: sensor placement 3} can be rewritten in the variable $x = (\alpha, \beta, z_1,...,z_m) \in \mathbb{R}^{m+3}$ as
\begin{equation}
\label{prob: SL3 not final}
\begin{split}
& \min_{x \in \mathbb{R}^{n+3}} \beta \\
& f_\kappa(\alpha, \beta, z_1,...,z_m)\le 0; \\
& \lambda_{\min}\Bigl( \sum_{i=1}^m z_i a_i a_i^T + \delta I \Bigr) \ge \alpha; \\
& (\mathbf{1},0)^T z = p;\quad g(x)\le 0, \\
& 0\le z_i \le 1,\quad i=1,...,m, \quad \alpha,\beta \ge \delta.
\end{split}
\end{equation}
where
\begin{equation}
\label{prob: functionss}
f_\kappa(\alpha, \beta, z_1,...,z_m) := -\alpha\beta + \lambda_{\max}\Bigl( \sum_{i=1}^m z_i a_i a_i^T + \delta I \Bigr).
\end{equation}
\end{lemma}

\begin{proof}
Under the constraint $\alpha \le \lambda_{\min}\bigl(\sum z_i a_i a_i^T + \delta I\bigr)$, we have
\[
\min_z \frac{\lambda_{\max}}{\lambda_{\min}}\Bigl( \sum_{i=1}^m z_i a_i a_i^T + \delta I \Bigr) = \min_{z,\alpha\ge\delta} \frac{\lambda_{\max}\bigl(\sum z_i a_i a_i^T + \delta I\bigr)}{\alpha}.
\]
Introducing the auxiliary variable $\beta \ge \delta$, the problem becomes
\begin{equation}
\label{eq: beta}
\min_{z,(\alpha,\beta\ge\delta)} \beta \quad\text{s.t.}\quad \beta \ge \frac{\lambda_{\max}\bigl(\sum z_i a_i a_i^T + \delta I\bigr)}{\alpha}.
\end{equation}
Since $\alpha \ge \delta$, the inequality $\beta \ge \frac{\lambda_{\max}}{\alpha}$ is equivalent to
\[
\beta\alpha \ge \lambda_{\max}\Bigl( \sum_{i=1}^m z_i a_i a_i^T + \delta I \Bigr).
\]
\end{proof}

The function $f_\kappa(\alpha, \beta, z_1,...,z_m)$ is weakly convex: $\lambda_{\max}(\cdot)$ is convex and $-\alpha\beta$ is 2-weakly convex (because $-\alpha\beta + \|\alpha,\beta\|^2$ is convex). Consequently, $f_\kappa + \|\cdot\|^2$ is convex. Moreover, the function $\alpha - \lambda_{\min}(\sum z_i a_i a_i^T + \delta I)$ is convex. Therefore, problem \eqref{prob: SL3 not final} can be recast in a form suitable for the cutting sphere algorithm, as shown in the next lemma.

\begin{lemma}
Problem \eqref{prob: SL3 not final} can be rewritten in the variable $x := (\alpha, \beta, z_1,...,z_m, x_{m+1}) \in \mathbb{R}^{m+4}$ as
\begin{equation}
\label{prob: SL3 final}
\tag{SL3.2}
\begin{split}
& \min_{x \in \mathbb{R}^{n+3}} \|x\|^2 \\
& f_\kappa(x)\le 0;\quad f_{\mathrm{obj}}(x)\le 0; \\
& \lambda_{\min}\Bigl( \sum_{i=1}^m z_i a_i a_i^T + \delta I \Bigr) \ge \alpha; \\
& g(x)\le 0, \\
& \sum_{i=1}^m z_i = p;\quad 0\le z_i \le 1,\quad i=1,...,m, \quad \alpha,\beta \ge \delta.
\end{split}
\end{equation}
where, letting $\overline{\lambda}_{\min}$ be an upper bound on the minimum eigenvalue of $\bigl(\sum z_i a_i a_i^T + \delta I\bigr)$ under the constraints of \eqref{prob: sensor placement 3}, and setting $\eta = \overline{\lambda}_{\min}^2 + k$, we have
\[
f_{\mathrm{obj}} = \beta^2 + \eta - \|x\|^2.
\]
\end{lemma}

\begin{proof}
Since $\beta \ge \delta$, we may replace the objective function in \eqref{prob: SL3 not final} by $\beta^2$. By \cite[Theorem 1]{bednarczuk2026outer}, it suffices to find $(\eta, \rho) \in \mathbb{R} \times \mathbb{R}_{++}$ such that
\[
D := \bigl\{ (\alpha,\beta,z_1,...,z_m) \in \arg\min \eqref{prob: SL3 not final} \ \big|\ 
\beta^2 + \eta \ge \tfrac{\rho}{2} \|(\alpha,\beta,z_1,...,z_m)\|^2 \bigr\} \neq \varnothing.
\]
Choosing $\rho = 2$, the condition becomes
\begin{align}
\beta^2 + \eta &\ge \|(\alpha,\beta,z_1,...,z_m)\|^2, \\
\eta &\ge \alpha^2 + \|z_1,...,z_m\|^2.
\end{align}
The choice $\eta = \overline{\lambda}_{\min}^2 + p^2$ satisfies both inequalities for every $(\alpha,\beta,z_1,...,z_m)$ in the argmin set of \eqref{prob: SL3 not final}. Indeed, $g(z_1,...,z_m) \le 0$ implies $z_i \in \{0,1\}$ and $\sum z_i^2 = p$. Moreover,
\[
\overline{\lambda}_{\min}^2 \ge \lambda_{\min}^2(\alpha,\beta,z_1,...,z_m) \ge \alpha^2,
\]
and therefore
\[
\eta = \overline{\lambda}_{\min}^2 + k = \overline{\lambda}_{\min}^2 + \|z_1,...,z_m\|^2 \ge \alpha^2 + \|z_1,...,z_m\|^2.
\]
\end{proof}

\begin{remark}
Since $-\lambda_{\min}(\cdot)$ is convex on the cone of positive definite matrices, an upper bound $\overline{\lambda}_{\min}$ can be computed by solving a convex optimization problem.
\end{remark}

		\section{Cutting sphere algorithms for sensor selection}

\subsection{Subgradients of the constraints}

Let \(A_i = a_i a_i^T\) for \(i = 1, \dots, m\) and define \(A(z) = \sum_{i=1}^m A_i z_i + \delta I\). In our setting, the vectors \(a_i \in \mathbb{R}^n\), \(i = 1, \dots, m\) are assumed to be linearly independent, and we select exactly \(p\) of the indices such that the corresponding \(z_i\) equal one. Thanks to the term \(\delta I\), $\delta>0$, the matrix \(A(z)\) is positive definite.

In the following we derive the subgradients of the functions \(-\log\det(\cdot)\), \(g + \rho \|\cdot\|^2\) and \(f_\kappa + \|\cdot\|^2\) introduced in the previous sections, evaluated at a point \(x \in \mathbb{R}^n\).

\begin{itemize}
\item \textbf{Gradient of \(-\log\det\)}.
For a symmetric positive definite matrix \(X \in \mathbb{R}^{n \times n}\), the function \(-\log\det(X)\) is convex. Consider the polyhedron in \eqref{eq:poly}

For every \(z_1\ge 0, \dots, z_m\ge 0\), we have that $A(z)$ is positive definite. The components of the gradient $-\nabla_z \log\det\Bigl( \sum_{i=1}^m z_i A_i + \delta I \Bigr)$ are as follows.
\[
-\nabla_z \log\det\Bigl( \sum_{i=1}^m z_i A_i + \delta I \Bigr)_i = -\operatorname{tr}(A_i A^{-1}(z)), \quad i=1,...,m,
\]
(see \cite[Example A.3]{Joshi2009}), and
\[
-\nabla_{x_{m+1}} \log\det\Bigl( \sum_{i=1}^m z_i A_i + \delta I \Bigr) = 0.
\]

\item \textbf{Subdifferential of \(g + \rho \|\cdot\|^2\)}. 
For \(\rho \ge 1\), and $x\in \mathbb{R}^m$. The components of an elelment $v$ of the convex subdifferential $\partial (g + \rho \|\cdot\|^2)(x)$ are,
for $i=1,...,m$,
\begin{equation}
 v_i =
\begin{cases}  2(\rho-1)x_i + 1 & \text{if } 0 < x_i < 1, \\
2(\rho + 1)x_i - 1 & \text{if } x_i < 0 \text{ or } x_i > 1, \\
2\rho x_i + \bigl\{ 
 \alpha(2x_i-1) + (1-\alpha)(-2x_i+1) \bigr.\ \ (\alpha \in [0,1])\} & \text{if } x_i \in \{0,1\}.\end{cases}
\end{equation}

\item \textbf{Subdifferential of \(f_\kappa + \|\cdot\|^2\)}. An element of the convex subdifferential of \(f_\kappa(\cdot) + \|\cdot\|^2\) at \(x \in \mathbb{R}^{n+3}\), $x=(\alpha, \beta, z_1,..,z_m, x_{m+1})$, $z_i\ge0$ $i=1,...,m$, is given by the vector 
\begin{equation}
\begin{pmatrix}
2\alpha - \beta,
2\beta - \alpha,
v_1,
\dots,
v_m,
2x_{m+1}
\end{pmatrix}^T,
\end{equation}
where the components \(v_i\), \(i=1,...,m\) are defined by
\[
v_i := 2z_i + \operatorname{trace}(G, a_i a_i^T)
\]
for some \(G \in \partial \lambda_{\max}(A)\) with \(A = z_i a_i a_i^T + \delta I\). If \(w\) is a leading eigenvector of \(A\), then \(G\) can be taken as \(ww^T\), see Section \ref{sec: not}.

On the other hand, an element of the subdifferential of the convex function \(-\lambda_{\min}(z_i a_i a_i^T + \delta I) + \alpha\) is
\begin{equation}
\begin{pmatrix}
1 ,
0 ,
v_1 ,
\dots ,
v_m ,
0
\end{pmatrix}^T,
\end{equation}
where \(v_i := -\operatorname{trace}(u u^T, a_i a_i^T)\) and \(u\) is an eigenvector corresponding to the smallest eigenvalue of \(z_i a_i a_i^T + \delta I\).
\end{itemize}

\subsection{Outer approximation set}

The construction of the outer approximation set for the feasible set of Problem \eqref{prob: SL final} is described in \cite[Definition 1]{bednarczuk2026outer}. Consider problem \eqref{prob: proj} and a point \(x_k\) with \(k \in \mathbb{N}\). Let \(I_k \subseteq \{1,...,m\}\) be the set of indices such that \(f_i(x_k) > 0\) for all \(i \in I_k\). Then, for \(a_i \ge \rho_i\), \(b_i \in \partial (f_i + a_i \|\cdot\|^2)(x_k)\) and \(c_i = a_i \|x_k\|^2 - b_i^T x_k + f_i(x_k)\), \(i \in I_k\), the outer approximation set is defined by
\begin{equation}
\label{outer set}
\mathcal{O}(x_k) := \bigl\{ x \in P \ \big|\ 
-a_i \|x\|^2 + b_i^T x + c_i \le 0,\ i \in I_k \bigr\},
\end{equation}
where \(P\) is the polyhedron given in \eqref{eq:poly}.

As an example, given \(\overline{x} \in E\) such that \(f_\Omega(\overline{x}) > 0\) and \(g(\overline{x}) > 0\), the outer approximation set for problem \eqref{prob: SL final} can be constructed as
\begin{equation}
\label{outer set ex}
\mathcal{O}(\overline{x}) := \bigl\{ x \in P \ \big|\ 
\phi_\Omega(x) - \phi_\Omega(\overline{x}) + f_\Omega(\overline{x}) \le 0;\ 
\phi_g(x) - \phi_g(\overline{x}) + g(\overline{x}) \le 0 \bigr\},
\end{equation}
where \(\phi_\Omega(x) := -a_\Omega \|x\|^2 + b_\Omega^T x\) with \(a_\Omega \ge 1\) larger than the modulus of weak convexity of \(f_\Omega\) (which equals 1), and \(b_\Omega \in \partial (f_\Omega + a_\Omega \|\cdot\|^2)(\overline{x})\). Similarly, \(\phi_g(x) := -a_g \|x\|^2 + b_g^T x\) with \(a_g \ge 1\) larger than the modulus of weak convexity of \(g\) (which equals 1), and \(b_g \in \partial (g + a_g \|\cdot\|^2)(\overline{x})\).

\subsection{Inexact cutting sphere algorithm}

Fix \(\varepsilon > 0\). Let \(k \in \mathbb{N}\). We say that iteration \(k+1\) is a *restart iteration* for the Inexact Cutting Sphere algorithm \cite[Algorithm 3]{bednarczuk2026outer} if
\begin{equation}
\label{it: restart}
x_{k+1} = \arg\min_{x \in \mathbb{R}^n} \|x\|^2 
\quad \text{s.t.} \quad 
x \in \mathcal{O}(x_k) \cap \bigl\{ x \in \mathbb{R}^n \ \big|\ \|x\|^2 \ge \|x_k\|^2 + \varepsilon \bigr\},
\end{equation}
where \(\mathcal{O}(x_k)\) is defined in \eqref{outer set}.

We say that iteration \(k+1\) is a *cumulative iteration* if
\begin{equation}
\label{it: cum}
x_{k+1} = \arg\min_{x \in \mathbb{R}^n} \|x\|^2 
\quad \text{s.t.} \quad 
x \in \mathcal{O}(x_k) \cap \Bigl( \bigcap_{j=r_k}^k \mathcal{O}(x_j) \Bigr) 
\cap \bigl\{ x \in \mathbb{R}^n \ \big|\ \|x\|^2 \ge \|x_{r_k-1}\|^2 + \varepsilon \bigr\},
\end{equation}
where \(r_k \le k\) denotes the index of the previous restart iteration.

\begin{lemma}
\label{lem: again}
Let \(\mathcal{A}\) be the feasible set of \eqref{prob: proj} and let \(k \in \mathbb{N}\). Then the following properties hold:
\begin{itemize}
    \item \(x_k \notin \mathcal{O}(x_k)\).
    \item \(\mathcal{A} \subset \mathcal{O}(x_k) \cap \bigcap_{j=r_k}^k \mathcal{O}(x_j)\).
\end{itemize}
\end{lemma}

\begin{proof}
\begin{itemize}
    \item Let \(I_k\) be the set of indices such that \(f_i(x_k) > 0\) for \(i \in I_k\), where the functions \(f_i\) describe the feasible set \(\mathcal{A}\) of \eqref{prob: proj}. Take any \(j \in I_k\) and consider the corresponding inequality \(-a_j \|x\|^2 + b_j^T x + c_j \le 0\) defining \(\mathcal{O}(x_k)\). By definition of \(c_j\) we have
    \[
    -a_j \|x_k\|^2 + b_j^T x_k + c_j = f_j(x_k) > 0,
    \]
    which shows that \(x_k \notin \mathcal{O}(x_k)\).

    \item For every \(x \in \mathcal{A}\) and every \(i \in I_p\), \(p \in \{r_k, \dots, k\}\), we have
    \[
    0 \ge f_i(x) \ge -a_i \|x\|^2 + b_i^T x + c_i,
    \]
    where the second inequality follows from \cite[Lemma 2]{bednarczuk2026outer}. Therefore \(\mathcal{A} \subset \mathcal{O}(x_p)\) for all \(p \in \{r_k, \dots, k\}\), which concludes the proof.
\end{itemize}
\end{proof}

Both problems \eqref{it: restart} and \eqref{it: cum} are quadratically constrained quadratic programs of the form
\begin{equation}
\label{outerproblem}
\tag{OP$_k$}
\minimize{x \in \OAP_k} \|x - z\|^2,
\end{equation}
where
\begin{equation}
\label{oaaset0}
\begin{split}
    &\OAP_k = P\cap \mathcal{O}(x_k) \cap \Bigl( \bigcap_{j=r_k}^k \mathcal{O}(x_j) \Bigr) 
\cap \bigl\{ x \in \mathbb{R}^n \ \big|\ \|x\|^2 \ge \|x_{r_k-1}\|^2 + \varepsilon \bigr\},\\
    &\bigl\{ x \in P \ \big|\ 
(\forall\, i \in \{1, \dots, m_k\})\ q_i^k(x) \le 0 \bigr\}
\end{split}
\end{equation}
and each \(q_i^k\) is a quadratic inequality of the form \(-a_i^k \|x\|^2 + (b_i^k)^T x + c_i^k\) with \(a_i^k \ge 0\).

The fact that we work with the polyhedron \(P\) instead of \(\mathbb{R}^n\) in \eqref{outer set} does not affect the analysis of \cite[Algorithm 1 and 3]{bednarczuk2026outer}, as shown in the following lemma.

\begin{lemma}
Let \(a > 0\) be such that \(\|x^*\|^2 > a\) for every global feasible solution \(x^*\) of \eqref{prob: proj}. Then the sequence \(\{x_k\}_{k \in \mathbb{N}}\) generated by \cite[Algorithm 1]{bednarczuk2026outer} (with \(\mathbb{R}^n\) replaced by \(P\)) lies in the compact set \(\operatorname{lev}_{\le a} \|\cdot\|^2 \cap P\) if $x_0\in P$. Under the assumptions that \(\mathcal{A} \neq \varnothing\) and that the functions \(f_i\), \(i=1,...,m\) are continuous on an open set containing \(\operatorname{lev}_{\le a} \|\cdot\|^2\), both \cite[Algorithm 1]{bednarczuk2026outer} and \cite[Algorithm 3]{bednarczuk2026outer} remain well-defined, and all theoretical results established in \cite{bednarczuk2026outer} continue to hold.
\end{lemma}

\begin{proof}
The sequence \(\{x_k\}_{k \in \mathbb{N}}\) lies in \(P\) by construction and in \(\operatorname{lev}_{\le a} \|\cdot\|^2\) because each \(\mathcal{O}_k\) is an outer approximation of \(\mathcal{A}\cap P\) by Lemma \ref{lem: again}. Let \(E := \operatorname{lev}_{\le a} \|\cdot\|^2 \cap P\). The set \(E\) is compact, and the functions \(f_i\) remain continuous on an open set containing \(E\). Therefore, the properties required in \cite[Assumption 1]{bednarczuk2026outer} are satisfied when \(\operatorname{lev}_{\le a} \|\cdot\|^2\) is replaced by \(\operatorname{lev}_{\le a} \|\cdot\|^2 \cap P\). Moreover, since \(P\) is defined by linear equalities and inequalities, problem \eqref{outerproblem} can still be solved using the method described in \cite[Subsection 8.2]{bednarczuk2026outer}.
\end{proof}

Notice that, if \(0 \notin P\), a suitable starting point \(x_0 \neq 0\) must be chosen satisfying:
\begin{enumerate}
    \item \(x_0 \in \operatorname{lev}_{\le a} \|\cdot\|^2 \cap P\),
    \item \(x_0 \notin \mathcal{A}\).
\end{enumerate}

\subsection{Output of the Inexact Cutting Sphere algorithm and parameters}

The main parameters of the Inexact Cutting Sphere algorithm \cite[Algorithm 3]{bednarczuk2026outer} are the tolerance \(\varepsilon > 0\) and the safety parameter \(\overline{m}\), which provides an upper bound on the number of constraints of problem \eqref{outerproblem}.

After choosing \(\varepsilon > 0\) and a starting point \(x_0 \in E\), the algorithm searches for solutions on the level set \(\operatorname{lev}_{=\|x_0\|^2} \|\cdot\|^2\). Ideally, if a global \(\varepsilon\)-solution \(x^*\) exists on this level set, the algorithm returns it. Otherwise, when the conditions of \cite[Algorithm 3]{bednarczuk2026outer} are met, it restarts from a point \(x_{r_1}\) lying on the level set \(\operatorname{lev}_{=\|x_0\|^2 + \varepsilon} \|\cdot\|^2\).

It may happen that, at some iteration \(k \in \mathbb{N}\), the number of constraints \(m_k\) of problem \eqref{outerproblem} exceeds the upper bound \(\overline{m}\), i.e., \(m_k > \overline{m}\). In this case the algorithm terminates without returning a solution. From our numerical experience, this situation occurs more frequently when the current level set is close to the optimal level set \(\operatorname{lev}_{=\|x^*\|^2} \|\cdot\|^2\). In such cases we recommend increasing \(\varepsilon\), or increasing \(\overline{m}\) if sufficient computational resources are available.

		
	\section{Inverse cutting sphere algorithm}
\label{sec: inv}

In the previous sections we observed that finding a good starting point for the cutting sphere algorithm can be difficult. Moreover, outer approximation algorithms (even in the convex case) often suffer from slow convergence when the initial point \(x_0 \in \mathbb{R}^n\) is far from a global minimizer \(x^* \in \mathbb{R}^n\).

In this section we propose a new method, called the **Inverse Cutting Sphere algorithm**. The algorithm starts from a known feasible solution \(\overline{x}\) of problem \eqref{prob: proj} satisfying \(\|\overline{x}\|^2 \ge \|x^*\|^2\), where \(x^*\) is a global solution of \eqref{prob: proj}. Its goal is either to find a better feasible point \(\widetilde{x}\) such that \(\|\widetilde{x}\|^2 = \|\overline{x}\|^2 - \varepsilon\), or to certify that \(\overline{x}\) is already a global \(\varepsilon\)-solution of \eqref{prob: proj}.

As in the inexact version, we introduce an upper bound \(\overline{m}\) on the computational cost (specifically, an upper bound on the number of constraints \(m_k\) of problem \eqref{outerproblem}). Before running the algorithm, we choose the tolerance \(\varepsilon > 0\) (which determines the target level \(\alpha = \|\overline{x}\|^2 - \varepsilon\)) and the safety parameter \(\overline{m}\). A suitable balance between \(\varepsilon\) and \(\overline{m}\) must be struck, since their difference represents the trade-off between solution accuracy and computational effort.

We prove that the Inverse Cutting Sphere algorithm generates a sequence \(\{x_k\}_{k \in \mathcal{L} \subset \mathbb{N}}\) satisfying the following properties:
\begin{enumerate}
    \item \(\mathcal{L}\) is finite; hence there exists an index \(l \in \mathcal{L}\) such that \(l > k\) for all \(k \in \mathbb{N} \setminus \{l\}\).
    \item If \(x_l\) is feasible for \eqref{prob: proj}, then it improves the heuristic solution \(\overline{x}\) by an amount \(\varepsilon\).
    \item If \(x_l\) is infeasible for \eqref{prob: proj} and the upper bound \(\overline{m}\) has not been reached, then \(\overline{x}\) is a global \(\varepsilon\)-solution of \eqref{prob: proj}.
    \item \(\|x_k\|^2 = \alpha\) for every \(k \in \mathcal{L}\).
\end{enumerate}

The name “Inverse” reflects the fact that we start from a feasible point and search for an improved solution on a *lower* level set. In contrast, the Inexact Cutting Sphere algorithm \cite[Algorithm 3]{bednarczuk2026outer} starts from an infeasible point and searches for a feasible point on successively *higher* level sets.

We begin with a preparatory lemma.

\begin{lemma} \cite[Lemma 12]{bednarczuk2026outer}
\label{lem: from out to feas}
Let \(k \in \mathbb{N}\) and let \(\alpha > 0\) be given with \(\alpha = \operatorname{Opt}(\eqref{outerproblem})\). Consider the outer approximation problem \eqref{outerproblem} with
\[
\mathcal{O}_k := \bigl\{ x \in P \ \big|\ 
L_i(\varphi_i^k, x_k)(x) = -a_i \|x\|^2 + b_{i,k}^\top x + c_{i,k} \le 0,\ 
i \in I(x_k) \bigr\} \cap \bigcap_{p=r_k}^{k-1} \mathcal{O}_p,
\]
where \(r_k\) is the index of the last restart iteration before \(k\).

Problem \eqref{outerproblem} is equivalent to the following auxiliary feasibility problem on the sphere \(\mathcal{S}_\alpha := \{ x \in \mathbb{R}^n \mid \|x\|^2 = \alpha \}\) and the polyhedron \(\mathcal{O}_k^\alpha\):
\begin{equation}
\label{Feasibility_prob}
\text{Find } x \in \mathcal{S}_\alpha \cap \mathcal{O}_k^\alpha,
\end{equation}
where
\[
\mathcal{O}_k^\alpha := \bigl\{ x \in P \ \big|\ 
b_{i,k}^\top x \le a_i \alpha - c_{i,k},\ 
i \in I(x_k) \bigr\} \cap \bigcap_{t=r_k}^{k-1} \mathcal{O}_t^\alpha,
\]
in the sense that \(x \in \mathcal{S}_\alpha \cap \mathcal{O}_k^\alpha\) if and only if \(x\) is a global optimal solution of \eqref{outerproblem}.
\end{lemma}

\subsection{Algorithm description and analysis}

In Algorithm \ref{alg: inv inexact} the following subroutine plays a central role.

\makeatletter
\renewcommand*{\ALG@name}{Subroutine}
\makeatother
\setcounter{algorithm}{0}

\begin{algorithm}[H]
\caption{Outer problem solver for given \(\varepsilon > 0\)}
\label{alg: outsolver}
\begin{itemize}
    \item Let \(\alpha = \|\overline{x}\|^2 - \varepsilon\). Define the binary variable \(\texttt{STOP} \in \{\texttt{True}, \texttt{False}\}\). Let \(\mathcal{S}_\alpha\) be the sphere of radius \(\alpha\) centered at the origin, and let the polyhedron \(\mathcal{O}_k^\alpha\) be defined as in Lemma \ref{lem: from out to feas}.
    
    \begin{enumerate}[label=\arabic*:]
        \item If \(\mathcal{S}_\alpha \cap \mathcal{O}_k^\alpha \neq \varnothing\), set \(x_{k+1} \in \mathcal{S}_\alpha \cap \mathcal{O}_k^\alpha\) and \(\texttt{STOP} = \texttt{FALSE}\).
        \item Else set \(x_{k+1} = x_k\) and \(\texttt{STOP} = \texttt{TRUE}\).
        \item \textbf{Return} \(x_{k+1}\), \(\texttt{STOP}\).
    \end{enumerate}
\end{itemize}
\end{algorithm}

Let \(m_k\) denote the computational cost of Subroutine \ref{alg: outsolver} (for example, the number of quadratic constraints in \eqref{outerproblem}).

\makeatletter
\renewcommand*{\ALG@name}{Algorithm}
\makeatother
\setcounter{algorithm}{0}

\begin{algorithm}[H]
\caption{Inverse inexact cutting sphere}
\label{alg: inv inexact}
\begin{algorithmic}[1]
\ENSURE{\(x_0 = \overline{x}\), \(k = 0\), \(\mathcal{O}_{-1} = \{x \in \mathbb{R}^n \mid \|x\|^2 \ge \alpha\}\), \(m_0 = 0\), \(\overline{m} > m\)}
\LOOP
    \IF{\(I(x_k) \neq \varnothing\) and \(m_k \le \overline{m}\) and \(\texttt{STOP} = \texttt{FALSE}\)}
        \STATE Build \(\mathcal{O}_k\) as in \eqref{it: cum}.
        \STATE Apply Subroutine \ref{alg: outsolver} to obtain \(x_{k+1}\) and update \(\texttt{STOP}\).
        \STATE \(k \leftarrow k + 1\)
    \ELSE
        \STATE \textbf{EXIT LOOP}
    \ENDIF
\ENDLOOP
\end{algorithmic}
\end{algorithm}

The following assumption is fundamental for the analysis of Algorithm \ref{alg: inv inexact}.

\begin{assumption}
\label{assumption:epsilon}
Let \(\overline{x}\) and \(x^*\) be a feasible point and a global optimal solution of Problem \eqref{prob: proj}, respectively. The parameter \(\varepsilon > 0\) satisfies:
If \(\|\overline{x}\|^2 - \varepsilon \ge \|x^*\|^2\), then \(\mathcal{A} \cap \operatorname{lev}_{=\alpha} \|\cdot\|^2 \neq \varnothing\), where \(\alpha = \|\overline{x}\|^2 - \varepsilon\).
\end{assumption}

The next lemma shows that Assumption \ref{assumption:epsilon} holds for every \(\varepsilon > 0\) when Problem \eqref{prob: proj} is the reformulation of Problem \eqref{prob: proj_pre} given in Theorem \ref{th: minnoncon}.

\begin{lemma}
\label{lem: ass epsilon holds}
Let Problem \eqref{prob: proj} be the reformulation of Problem \eqref{prob: proj_pre} described in Theorem \ref{th: minnoncon}. Then Assumption \ref{assumption:epsilon} holds for every \(\varepsilon > 0\).
\end{lemma}

\begin{proof}
Let \(\varepsilon > 0\) and let \(\overline{x}\) be a feasible point of Problem \eqref{prob: proj}. Fix \(\varepsilon_1 \in (0, \varepsilon)\). By Theorem \ref{th: minnoncon}, the feasible set of \eqref{prob: proj} is
\[
\mathcal{A} = \overline{\mathcal{A}} \cap \{ x \in \mathbb{R}^n \mid f_1(x) \le 0 \}.
\]
Let \(x^*\) be a global solution of \eqref{prob: proj} and let \(v > 0\) satisfy \(\alpha - v = \|x^*\|^2\). From \eqref{eq: new constr form} we have
\[
f_1(x^*) = F(x_1^*, \dots, x_{n-1}^*) + \eta - \frac{\rho}{2} \|x^*\|^2 \le 0,
\]
where \((\rho, \eta) \in ]0, +\infty[ \times \mathbb{R}\). Note that the set $\overline{\mathcal{A}}$ depends only on the first \(n-1\) components.

Define
\[
\tilde{x} := (x_1^*, \dots, x_{n-1}^*, x_n^* + \omega),
\]
where \(\omega \in \mathbb{R}\) solves \(\omega^2 + 2\omega x_n^* = v\), i.e.,
\[
\omega = -x_n^* + \sqrt{(x_n^*)^2 + v} > 0.
\]
Then
\[
\|\tilde{x}\|^2 = \|x^*\|^2 + \omega^2 + 2\omega x_n^* = \|x^*\|^2 + v \le \alpha.
\]
Clearly \(\tilde{x} \in \overline{\mathcal{A}}\). Substituting into the constraint gives
\[
f_1(\tilde{x}) = f_1(x^*) - \frac{\rho}{2} v \le f_1(x^*) \le 0.
\]
Consequently, if there exists \(v > 0\) such that \(\|x^*\|^2 + v = \alpha\), then \(\mathcal{A} \cap \operatorname{lev}_{=\alpha} \|\cdot\|^2 = \varnothing\). In particular, whenever \(\alpha = \|\overline{x}\|^2 - \varepsilon \ge \|x^*\|^2\), such a \(v > 0\) always exists.
\end{proof}

Algorithm \ref{alg: inv inexact} generates the sequence \((x_k)_{k \in \widehat{\mathbb{L}}}\), where
\[
\widehat{\mathbb{L}} := \{ k \in \mathbb{N} : I(x_k) \neq \emptyset \text{ and } m_k \le \overline{m} \}.
\]

The following lemma characterises the output of Subroutine \ref{alg: outsolver}.

\begin{lemma}
\label{lem: sub1 outcome}
Let Assumption \ref{assumption:epsilon} hold. For any \(k \in \widehat{\mathbb{L}}\), let \(\overline{x}\) be the starting heuristic solution and let \(\operatorname{Opt}(\eqref{outerproblem})\) denote the optimal value of the outer approximation problem \eqref{outerproblem}.
\begin{enumerate}
    \item If Subroutine \ref{alg: outsolver} returns \(\texttt{STOP} = \texttt{FALSE}\), then \(x_{k+1}\) is a global solution of \eqref{outerproblem} and \(\operatorname{Opt}(\eqref{outerproblem}) = \|\overline{x}\|^2 - \varepsilon\).
    \item If Subroutine \ref{alg: outsolver} returns \(\texttt{STOP} = \texttt{TRUE}\), then \(\operatorname{Opt}(\eqref{prob: proj}) > \|\overline{x}\|^2 - \varepsilon\).
\end{enumerate}
\end{lemma}

\begin{proof}
\begin{enumerate}
    \item If \(\texttt{STOP} = \texttt{FALSE}\), the claim follows directly from Lemma \ref{lem: from out to feas}.
    \item If \(\texttt{STOP} = \texttt{TRUE}\), then the feasibility problem in Lemma \ref{lem: from out to feas} has no solution, i.e., \(\mathcal{S}_\alpha \cap \mathcal{O}_k^\alpha = \varnothing\). Hence \(\operatorname{Opt}(\eqref{outerproblem}) \neq \alpha\) by Lemma \ref{lem: from out to feas}.

    Suppose for contradiction that \(\operatorname{Opt}(\eqref{prob: proj}) \le \alpha\). By Assumption \ref{assumption:epsilon} there exists \(x \in \mathcal{A}\) with \(\|x\|^2 = \alpha\). Then \(x \in \mathcal{A} \cap \mathcal{S}_\alpha\). From the definition of the feasible set of \eqref{outerproblem} and Lemma \ref{lem: again} we obtain \(x \in \mathcal{O}_k \cap \mathcal{S}_\alpha = \mathcal{O}_k^\alpha \cap \mathcal{S}_\alpha\), which contradicts the emptiness of the intersection. 
\end{enumerate}
\end{proof}

\begin{proposition}
\label{prop:inexactbis}
Consider Problem \eqref{prob: proj}. Let \(\overline{x}\) be a feasible heuristic solution, let \(\varepsilon > 0\) and set \(\alpha = \|\overline{x}\|^2 - \varepsilon\). Let \(\overline{m} \in \mathbb{N}\) and suppose that Assumption \ref{assumption:epsilon} holds. Let \((\mathcal{O}_{k-1}, x_k)_{k \in \mathbb{L}}\) be the sequence generated by Algorithm \ref{alg: inv inexact}.
\begin{enumerate}
    \item For every \(k \in \mathbb{L}\) we have \(\mathcal{O}_{k-1} \subseteq \mathcal{O}_{k-2} \subseteq \cdots \subseteq \mathcal{O}_0\).
    \item The sequence \(\{\|x_k\|^2\}_{k \in \mathbb{L}}\) is constant and equal to \(\alpha\).
    \item For every \(k \in \mathbb{L}\), \(I(x_k) = \varnothing\) if and only if \(x_k \in \mathcal{A}\) and \(\|x_k\|^2 = \alpha\).
    \item For every \(k \in \mathbb{L}\), the number of constraints in \(\mathcal{O}_{k-1}\) never exceeds \(\overline{m}\).
    \item There exists an iteration \(k^* \in \mathbb{N}\) at which Algorithm \ref{alg: inv inexact} terminates.
    \item We have \(x_{k^*} \in \operatorname{lev}_\alpha \|\cdot\|^2\), and one of the following holds:
    \begin{itemize}
        \item \(x_{k^*} \in \mathcal{A}\), i.e., \(x_{k^*}\) is feasible and satisfies \(\|x_{k^*}\|^2 = \|\overline{x}\|^2 - \varepsilon\);
        \item \(m_{k^*} > \overline{m}\);
        \item \(\overline{x}\) is an \(\varepsilon\)-global solution of Problem \eqref{prob: proj}.
    \end{itemize}
\end{enumerate}
\end{proposition}

\begin{proof}
\begin{enumerate}
    \item Follows from the fact that Algorithm \ref{alg: inv inexact} accumulates constraints.
    \item Subroutine \ref{alg: outsolver} applied at iteration \(k-1\) either returns a point with \(\|x_k\|^2 = \alpha\) or sets \(\texttt{STOP} = \texttt{TRUE}\).
    \item The set \(I(x_k)\) contains the indices of the constraints \(f_i\) that are violated at \(x_k\). Hence \(I(x_k) = \varnothing\) if and only if \(x_k \in \mathcal{A}\). Moreover, \(\|x_k\|^2 = \alpha\) by the previous item.
    \item This follows directly from the construction of the algorithm.
    \item If Subroutine \ref{alg: outsolver} returns \(\texttt{STOP} = \texttt{TRUE}\) at iteration \(k^*-1\), the algorithm stops at \(k^*\). If \(\texttt{STOP} = \texttt{FALSE}\) for all iterations and there exists some \(k^*\) with \(I(x_{k^*}) = \varnothing\), the algorithm also stops. Otherwise, since \(x_k \notin \mathcal{A}\) for all \(k \in \mathbb{L}\), the number of constraints satisfies \(m_{k+1} \ge m_k + 1\) by \eqref{it: cum}. Hence there must exist an iteration where \(m_{k^*} > \overline{m}\), at which point the algorithm terminates.
    \item The equality \(\|x_{k^*}\|^2 = \|\overline{x}\|^2 - \varepsilon\) follows from item (ii). The algorithm stops when either \(m_{k^*} > \overline{m}\), \(I(x_{k^*}) = \varnothing\), or \(\texttt{STOP} = \texttt{TRUE}\). In the case \(I(x_{k^*}) = \varnothing\), item (iii) implies that \(\|\overline{x}\|^2 - \varepsilon = \|x^*\|^2\). When \(\texttt{STOP} = \texttt{TRUE}\), Lemma \ref{lem: sub1 outcome} shows that \(\overline{x}\) is an \(\varepsilon\)-global solution.
\end{enumerate}
\end{proof}

When \(\|\overline{x}\|^2 - \varepsilon < \operatorname{Opt}(\eqref{prob: proj})\), the point \(\overline{x}\) is already an \(\varepsilon\)-solution, but it may not be possible to certify this if the number of constraints exceeds \(\overline{m}\) at termination. The next proposition shows that there always exists a sufficiently large \(\overline{m}\) for which \(\texttt{STOP} = \texttt{TRUE}\) (which, by Lemma \ref{lem: sub1 outcome}, certifies that \(\overline{x}\) is an \(\varepsilon\)-global solution).

\begin{proposition}
\label{always}
Consider Problem \eqref{prob: proj}. Let \(\overline{x}\) be a feasible heuristic solution and let \(\varepsilon > 0\). Define \(\alpha := \|\overline{x}\|^2 - \varepsilon\). Suppose that Assumption \ref{assumption:epsilon} holds and that there exists \(\tau \ge 0\) with \(\|\overline{x}\|^2 - \varepsilon < \tau\) such that the functions \(f_i\), \(i=1,...,m\) are continuous on an open set containing \(\operatorname{lev}_{\le \tau} \|\cdot\|^2\).

Let \((\mathcal{O}_{k-1}, x_k)_{k \in \mathbb{L}}\) be generated by Algorithm \ref{alg: inv inexact} and assume that \(\|\overline{x}\|^2 - \varepsilon < \operatorname{Opt}(\eqref{prob: proj})\). Then there exists \(k^* \in \mathbb{N}\) such that \(\texttt{STOP} = \texttt{TRUE}\) at iteration \(k^* - 1\).
\end{proposition}

\begin{proof}
Algorithm \ref{alg: inv inexact} performs only cumulative iterations. Suppose for contradiction that \(\texttt{STOP} = \texttt{FALSE}\) for all \(k \in \mathbb{L}\). Then, for every \(k \in \mathbb{L}\), the point \(x_{k+1}\) produced by Subroutine \ref{alg: outsolver} lies in \(\mathcal{S}_\alpha \cap \mathcal{O}_k^\alpha\) and is therefore a global solution of \eqref{outerproblem} by Lemma \ref{lem: from out to feas}. Consequently, the sequence generated by Algorithm \ref{alg: inv inexact} coincides with that of \cite[Algorithm 1]{bednarczuk2026outer}.

However, no iterate \(x_{k+1}\) can belong to \(\mathcal{A}\), because
\[
\|x_{k+1}\|^2 = \alpha = \|\overline{x}\|^2 - \varepsilon < \operatorname{Opt}(\eqref{prob: proj}) = \|x^*\|^2
\]
for any global solution \(x^*\) of \eqref{prob: proj}. Thus \(\mathbb{L} = \mathbb{N}\) and \(\{x_k\}_{k \in \mathbb{L}}\) is infinite. By \cite[Theorem 2]{bednarczuk2026outer} there exists a convergent subsequence whose limit is a global solution of \eqref{prob: proj}. This contradicts the fact that every \(x_{k+1}\) lies in \(\mathcal{S}_\alpha \cap \mathcal{O}_k^\alpha\), since \(\mathcal{S}_\alpha = \{x \in \mathbb{R}^n \mid \|x\|^2 = \alpha\}\).
\end{proof}

\begin{remark}
\label{rem}
If Algorithm \ref{alg: inv inexact} returns a feasible solution \(\tilde{x}\) satisfying \(\|\tilde{x}\|^2 = \|\overline{x}\|^2 - \varepsilon\), we may restart the algorithm with \(\overline{x} := \tilde{x}\) in order to search for an even better solution.
\end{remark}

				\section{Experiments}
	\label{sec: exp}
	
	\subsection{Explanation of the Dataset Structure}
\label{sec:dataset}

The datasets are stored in \texttt{.npz} files. These files contain hundreds of virtual wind-tunnel tests performed with the program XFOIL \cite{Drela1989} using Mach $0.1$ and Reynolds number $200000$. The program simulated the airflow over the airfoils NACA 2412, NACA 2418, NACA 0012 and NACA 0018 at various angles of attack, recorded the pressure distribution on the surface, and saved the results in a clean, structured format suitable for machine learning and further analysis.

For the training dataset we consider 500 equally spaced angles of attack (in degrees) ranging from \(-7.0^\circ\) to \(+7.0^\circ\). For the test dataset we use 100 equally spaced angles of attack in the same range.

\begin{itemize}
    \item \textbf{Shape}:
    \begin{itemize}
        \item Training file: \(500\) rows \(\times 160\) columns
        \item Test file: \(100\) rows \(\times 160\) columns
    \end{itemize}
    \item Each \textbf{row} corresponds to one specific angle of attack.
    \item Each \textbf{column} corresponds to one of the 160 surface points on the airfoil.
    \item The stored values are the \textbf{pressure coefficient} (\(C_p\)) at the respective surface points.
\end{itemize}

In short, for every tested angle of attack the program recorded the pressure distribution across the 160 surface points.

Each airfoil is associated with two files. For the NACA 0012 airfoil, for example, we have:
\begin{itemize}
    \item \texttt{naca0012\_train\_500.npz} (used for training)
    \item \texttt{naca0012\_test\_100.npz} (used for testing and validation)
\end{itemize}

\subsection{Sensor selection procedure (training process)}
\label{sec:train}

The training process takes the training dataset as input and returns the optimal sensor positions as output.

The pressure distribution on each airfoil is represented by XFOIL through 160 discrete measurements corresponding to 160 possible sensor locations. We assume that a sensor can be placed at any of these 160 points. The sensor configuration is encoded by a binary vector \(z \in \mathbb{R}^{160}\) whose entries satisfy \(z_j = 1\) if a sensor is placed at location \(j\) and \(z_j = 0\) otherwise.

As described in Section \ref{sec: full procedure}, the training procedure consists of three phases:
\begin{enumerate}
    \item Apply proper orthogonal decomposition (POD) to the training dataset (as described in \cite{manohar2017data}, also known as principal component analysis). This produces a new matrix \(A \in \mathbb{R}^{p \times 160}\), where \(p\) is the prescribed number of sensors.
    
    \item Formulate a nonconvex optimisation problem to determine the optimal sensor locations, following the approach of Section \ref{sec: reform}. The objective is to select \(p\) columns of \(A\) that form the matrix \(A_p \in \mathbb{R}^{p \times p}\). Two models are considered: one that minimises \(-\log\det(A_p)\) and one that minimises the condition number of \(A_p\).
    
    \item Solve the resulting nonconvex problem to \(\varepsilon\)-global optimality using the Inexact Cutting Sphere algorithm \cite[Algorithm 3]{bednarczuk2026outer}, or start from a heuristic solution \(\overline{z}\) (e.g., obtained by QDEIM) and improve it with Algorithm \ref{alg: inv inexact}.
\end{enumerate}

\subsection{Test signal reconstruction (validation)}

Using the notation of Section \ref{sec: full procedure}, let \(\omega_s \in \mathbb{R}^{160}\), \(s = 1, \dots, 100\), denote the pressure distributions (signals) contained in the test dataset. Validation consists in taking the sensor measurements \(y_s \in \mathbb{R}^p\) and reconstructing the full signals \(\omega_s\) via the least-squares method described in Section \ref{sec: full procedure}.

\subsection{Results for the Inexact Cutting Sphere algorithm}

We first test the Inexact Cutting Sphere algorithm \cite[Algorithm 3]{bednarczuk2026outer} on problem \eqref{prob: SL final}. After determining the optimal sensor positions for a given training dataset of a particular NACA airfoil, we reconstruct the 100 pressure distributions of the corresponding test set, obtaining the approximations \(\overline{\omega}_s\), \(s = 1, \dots, 100\).

The total reconstruction error is defined as
\[
\text{error tot CS} = \sum_{s=1}^{100} \|\omega_s - \overline{\omega}_s\|.
\]

For comparison, we also consider the heuristic sensor positions obtained with the QDEIM method in the reduced-order modelling setting of \cite{manohar2017data}. For each airfoil we record:
\begin{itemize}
    \item The sensor positions \(\tilde{z} \in \mathbb{R}^{160}\) returned by QDEIM for 3 and 5 sensors: $[55, 82, 96]$, $[14, 28, 56, 82, 96]$ for NACA 2418, $[63, 80, 99]$ for NACA 0018, $[61, 81, 89]$ and $[ 62 , 80 , 90, 100, 102]$ for NACA 2412, $[71, 80, 93]$ and $[ 57,  70,  80, 102, 159]$ for NACA 0012.
    \item The matrix \(A_{\text{QDEIM}}\) formed by the corresponding columns of the POD basis matrix \(A\).
    \item The reconstructed pressure distributions \(\tilde{\omega}_s\), \(s = 1, \dots, 100\).
    \item The vector \(\text{VS QDEIM} = [a, b]\), where \(a = -\log\det(A_{\text{QDEIM}})\) (or the condition number of \(A_{\text{QDEIM}}\)) and
    \[
    b = \sum_{s=1}^{100} \max\bigl\{0,\ \|\omega_s - \tilde{\omega}_s\|^2 - \|\omega_s - \overline{\omega}_s\|^2\bigr\}.
    \]
    The integer \(b\) counts how many test snapshots are reconstructed more accurately by the cutting-sphere solution than by QDEIM.
    \item The total reconstruction error \(\text{error tot QDEIM} = \sum_{s=1}^{100} \|\omega_s - \tilde{\omega}_s\|\).
\end{itemize}

All experiments were performed on a Microsoft Windows 10 Pro workstation equipped with 128 GB of DDR4 RAM and an AMD Ryzen 9 3950X 16-core processor running at 3.5 GHz.

The cutting-sphere algorithm was initialised from a feasible point \(x_0\) satisfying \(\|x_0\|^2 = p\), where \(p\) is the number of sensors. For each model (\eqref{prob: SL final} or \eqref{prob: SL3 final}) we report the objective value, the Euclidean norm of the solution, and the computational time. The results are summarised in the tables below.

\begin{table}[htbp]
\centering
\small
\caption{Optimal positions for 3 sensors and results obtained with \cite[Algorithm 3]{bednarczuk2026outer} applied to \eqref{prob: SL final}.}
\label{tab:sensor-placement 3 logdet}
\begin{tabular}{l | c | l | r | r | r | r | r}
\hline
Airfoil & \(\varepsilon\) & Sensors Position & \(-\log\det\) & \(\|\cdot\|^2\) & Time & VS QDEIM & error tot CS \\
 &  & by Cutting Sphere &   &  &  & (obj.; score) & / QDEIM \\
\hline
2418 & 0.01 & [51, 78, 93] & 9.785687 & 9.01 & 53 s & 9.884290; 69 & 30.636 / 30.648 \\
2418 & 0.10 & [51, 79, 94] & 9.786872 & 9.10 & 16 s & 9.884290; 66 & 30.639 / 30.648 \\
0018 & 0.10 & [60, 79, 98] & 9.874790 & 9.10 & 30 s & 9.882912; 75 & 30.443 / 30.450 \\
0018 & 0.01 & [60, 79, 98] & 9.874790 & 9.10 & 35 s & 9.882912; 75 & 30.443 / 30.450 \\
0012 & 0.01 & [69, 80, 92] & 8.972200 & 9.01 & 30 s & 9.022418; 47 & 28.343 / 28.310 \\
0012 & 0.001 & [69, 80, 92] & 8.972200 & 9.004 & 56 s & 9.022418; 47 & 28.343 / 28.310 \\
2412 & 0.10 & [55, 78, 87] & 8.945973 & 9.10 & 2 s & 9.109868; 95 & 32.864 / 34.609 \\
2412 & 0.01 & [56, 78, 87] & 8.945817 & 9.01 & 16 s & 9.022418; 96 & 32.882 / 34.609 \\
\hline
\end{tabular}
\end{table}

\begin{table}[htbp]
\centering
\small
\caption{Optimal positions for 5 sensors and results obtained with \cite[Algorithm 3]{bednarczuk2026outer} applied to \eqref{prob: SL final}.}
\label{tab:sensor-placement 5 logdet}
\begin{tabular}{l | c | l | r | r | r | r | r}
\hline
Airfoil & \(\varepsilon\) & Sensors Position & \(-\log\det\) & \(\|\cdot\|^2\) & Time & VS QDEIM & error tot \\
 &  & by Cutting Sphere &   &  &  & (obj.; score) & \\
\hline
2418 & 0.10 & [13, 27, 55, 81, 95] & 14.785104 & 25.1 & 2025 s & 14.851554; 61 & 24.094 / 24.720 \\
0018 & 0.10 & [8, 62, 80, 99, 151] & 15.564481 & 25.1 & 659 s & 15.585886; 46 & 23.874 / 23.978 \\
0012 & 0.10 & [0, 57, 75, 85, 102] & 14.616239 & 25.4 & 1370 s & 14.908424; 57 & 30.396 / 32.930 \\
2412 & 0.10 & [62, 80, 90, 100, 102] & 11.827939 & 25.1 & 175 s & 11.827939; -- & 24.016 / 24.016 \\
\hline
\end{tabular}
\end{table}

\begin{table}[htbp]
\centering
\small
\caption{Optimal positions for 3 sensors and results obtained with \cite[Algorithm 3]{bednarczuk2026outer} applied to \eqref{prob: SL3 final}.}
\label{tab:sensor-placement 3 condition number}
\begin{tabular}{l | c | l | r | r | r | r | r}
\hline
Airfoil & \(\varepsilon\) & Sensors Position & Cond. Number & \(\|\cdot\|^2\) & Time & VS QDEIM & error tot \\
 &  & by Cutting Sphere &   &  &  & (obj.; score) & \\
\hline
2418 & 5 & [49, 77, 96] & 1.99206 & 8 & 4 s & 2.57215; 65 & 30.639 / 30.648 \\
0018 & 5 & [57, 80, 102] & 1.962025 & 8 & 86 s & 2.687235; 95 & 30.430 / 30.450 \\
0012 & 10 & [62, 79, 95] & 2.930838 & 13 & 258 s & 5.073709; 66 & 28.278 / 28.310 \\
2412 & 20 & [55, 78, 89] & 3.259366 & 23 & 348 s & 4.266402; 95 & 33.887 / 34.609 \\
\hline
\end{tabular}
\end{table}

The results in Tables \ref{tab:sensor-placement 3 logdet}--\ref{tab:sensor-placement 3 condition number} demonstrate that the Inexact Cutting Sphere algorithm \cite[Algorithm 3]{bednarczuk2026outer} consistently outperforms the QDEIM heuristic, both in terms of the objective value of the sensor-selection problem and in reconstruction accuracy.
The only exception occurs for the NACA 2412 airfoil with 5 sensors (Table \ref{tab:sensor-placement 5 logdet}), where both methods return the same sensor positions.

These experiments also confirm that the QDEIM heuristic already produces solutions of high quality. 

The cutting-sphere approach is computationally expensive. It is therefore essential to balance the desired precision \(\varepsilon\) against the computational budget. In our implementation we terminate the algorithm when the number of constraints of problem \eqref{outerproblem} exceeds the safety limit of 3000, i.e. we set $\overline{m}=3000$.

In the following experiment we instead attempt to improve a known feasible solution using the Inverse Cutting Sphere algorithm (Algorithm \ref{alg: inv inexact}).

\subsection{Results for the Inverse Cutting Sphere algorithm}

We test Algorithm \ref{alg: inv inexact} on the NACA 2412 dataset using problem \eqref{prob: SL final} with 5 sensors. As starting point we take the feasible solution \(\overline{x}\) obtained by the Inexact Cutting Sphere algorithm in Table \ref{tab:sensor-placement 5 logdet}, which satisfies \(\|\overline{x}\|^2 = 25.1\) (which could not improve the solution found by the QDEIM heuristic).

Note that the first 160 components of \(\overline{x}\) are binary (indicating sensor placement), while the last component has a different meaning. Consequently, the initial point \(x_0\) for the Inverse Cutting Sphere algorithm (Algorithm \ref{alg: inv inexact}) has its first 160 entries identical to those of \(\overline{x}\), and its last entry is chosen so that
\[
\|x_0\|^2 = \|\overline{x}\|^2 - \varepsilon = 25.1 - \varepsilon.
\]

We performed the following three tests:

\begin{enumerate}
    \item With \(\varepsilon = 0.1\) the algorithm terminated after 21 iterations (9 seconds) and correctly certified that \(\overline{x}\) is already a global \(0.1\)-optimal solution, confirming the result of the Inexact Cutting Sphere algorithm.
    
    \item With \(\varepsilon = 0.09\) the algorithm returned a feasible point satisfying \(\|x_0\|^2 = 25.01\). While this constitutes an improvement over the \(\varepsilon = 0.1\) solution, the sensor positions remained unchanged because of the specific construction of the initial point.
    
    \item With \(\varepsilon = 0.095\), Algorithm \ref{alg: inv inexact} discovered a new feasible solution whose properties are reported in Table \ref{tab:INV}.
\end{enumerate}

\begin{table}[htbp]
\centering
\small
\caption{Optimal sensor positions found by the Inverse Cutting Sphere algorithm (Algorithm \ref{alg: inv inexact}) on the NACA 2412 dataset with \(\varepsilon = 0.095\).}
\label{tab:INV}
\begin{tabular}{l | r | r | r | r | r}
\hline
Sensors Position & \(-\log\det\) & \(\|\cdot\|^2\) & Time & VS QDEIM & error tot \\
 &  &  &  & (obj.; score) & \\
\hline
[62, 80, 91, 100, 102] & 11.822875 & 25.005 & 9 s & 11.827939; 55 & 23.983 / 24.016 \\
\hline
\end{tabular}
\end{table}

\begin{figure*}
    \raggedright
    \begin{minipage}[t]{0.48\textwidth}
        \centering
        \includegraphics[width=\textwidth, height=0.34\textheight, keepaspectratio]{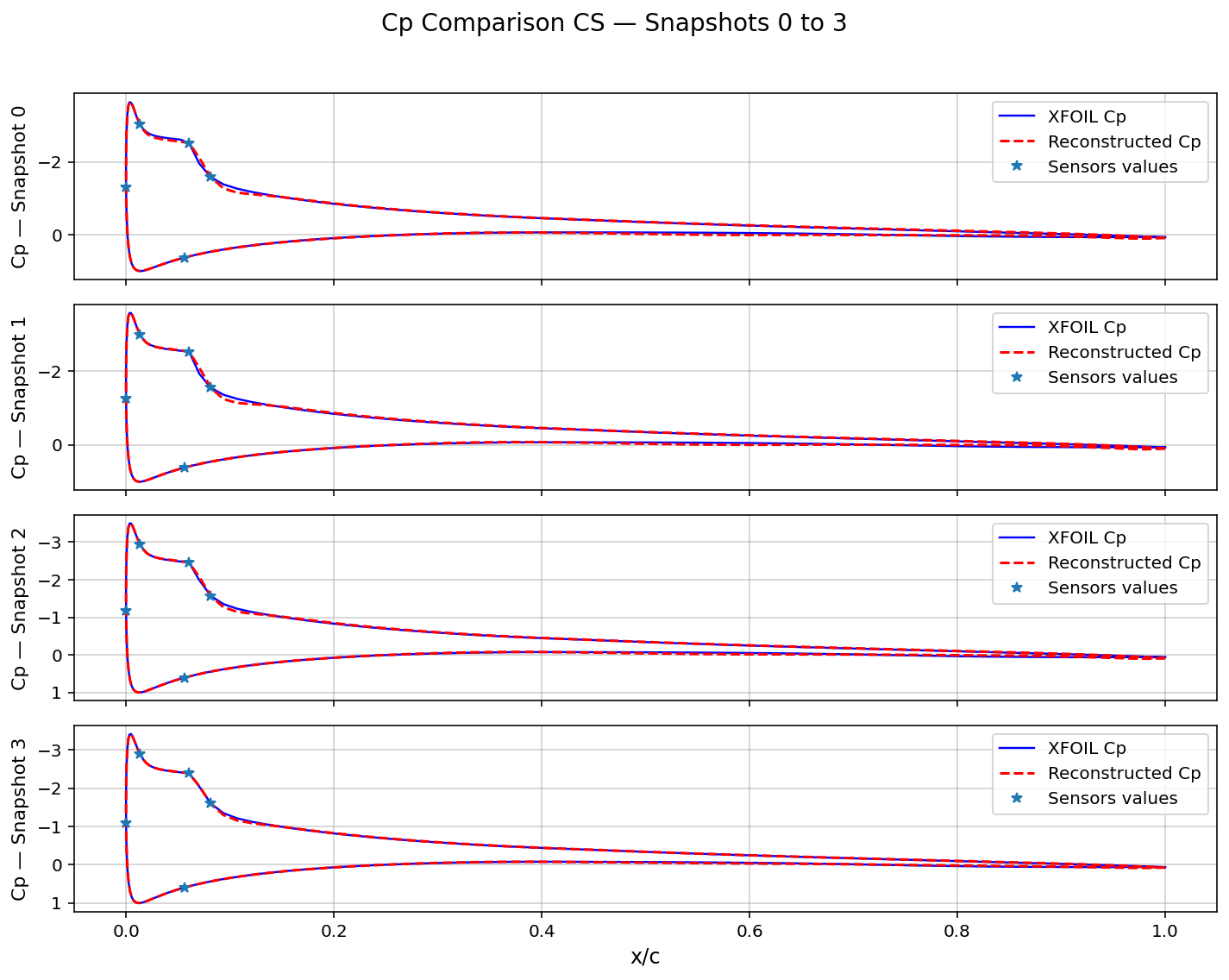}
    \end{minipage}
    \hfill
    \begin{minipage}[t]{0.48\textwidth}
        \centering
        \includegraphics[width=\textwidth, height=0.34\textheight, keepaspectratio]{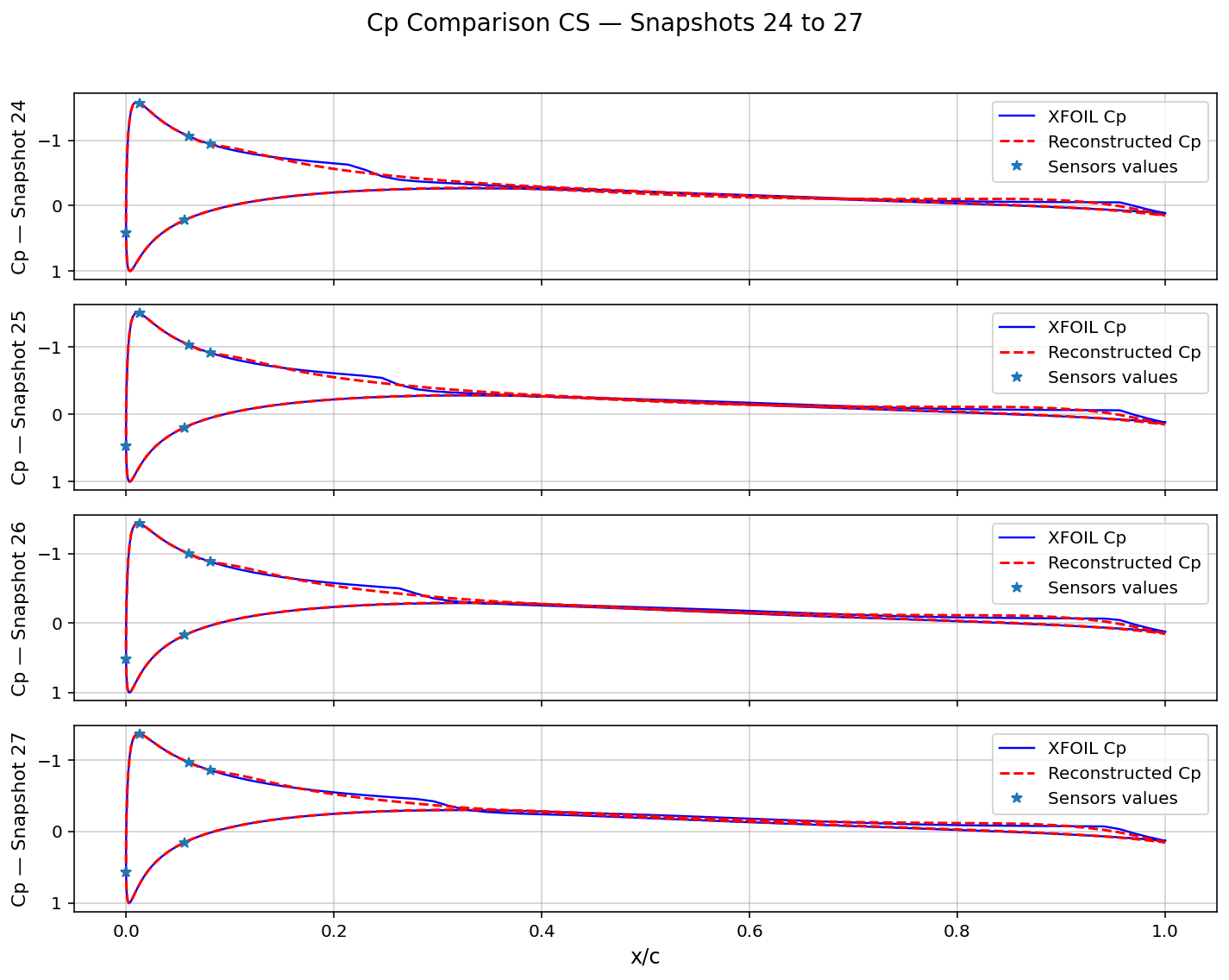}
    \end{minipage}
    
    \vspace{0.7cm}
    
    \begin{minipage}[t]{0.48\textwidth}
        \centering
        \includegraphics[width=\textwidth, height=0.34\textheight, keepaspectratio]{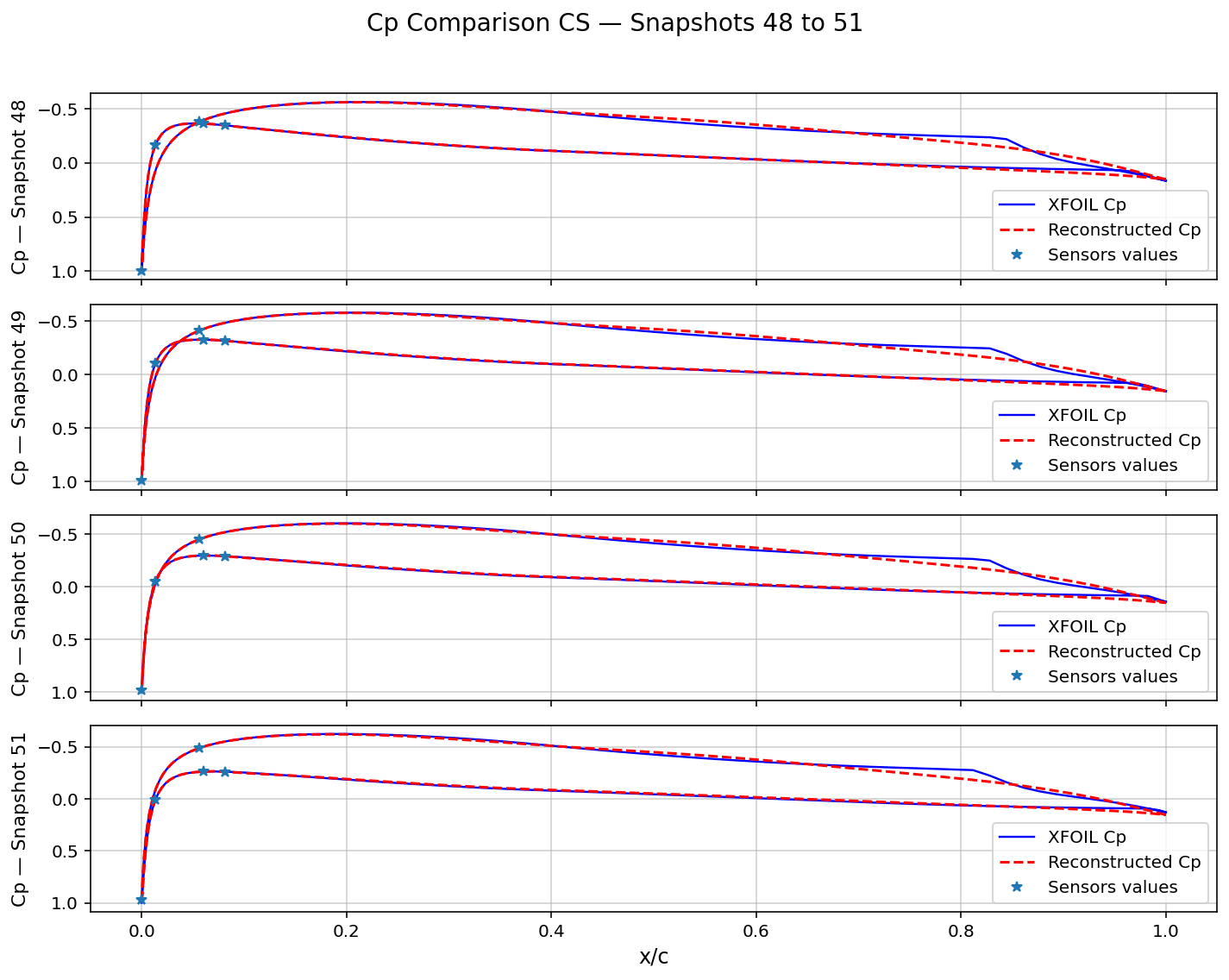}
    \end{minipage}
    \hfill
    \begin{minipage}[t]{0.48\textwidth}
        \centering
        \includegraphics[width=\textwidth, height=0.34\textheight, keepaspectratio]{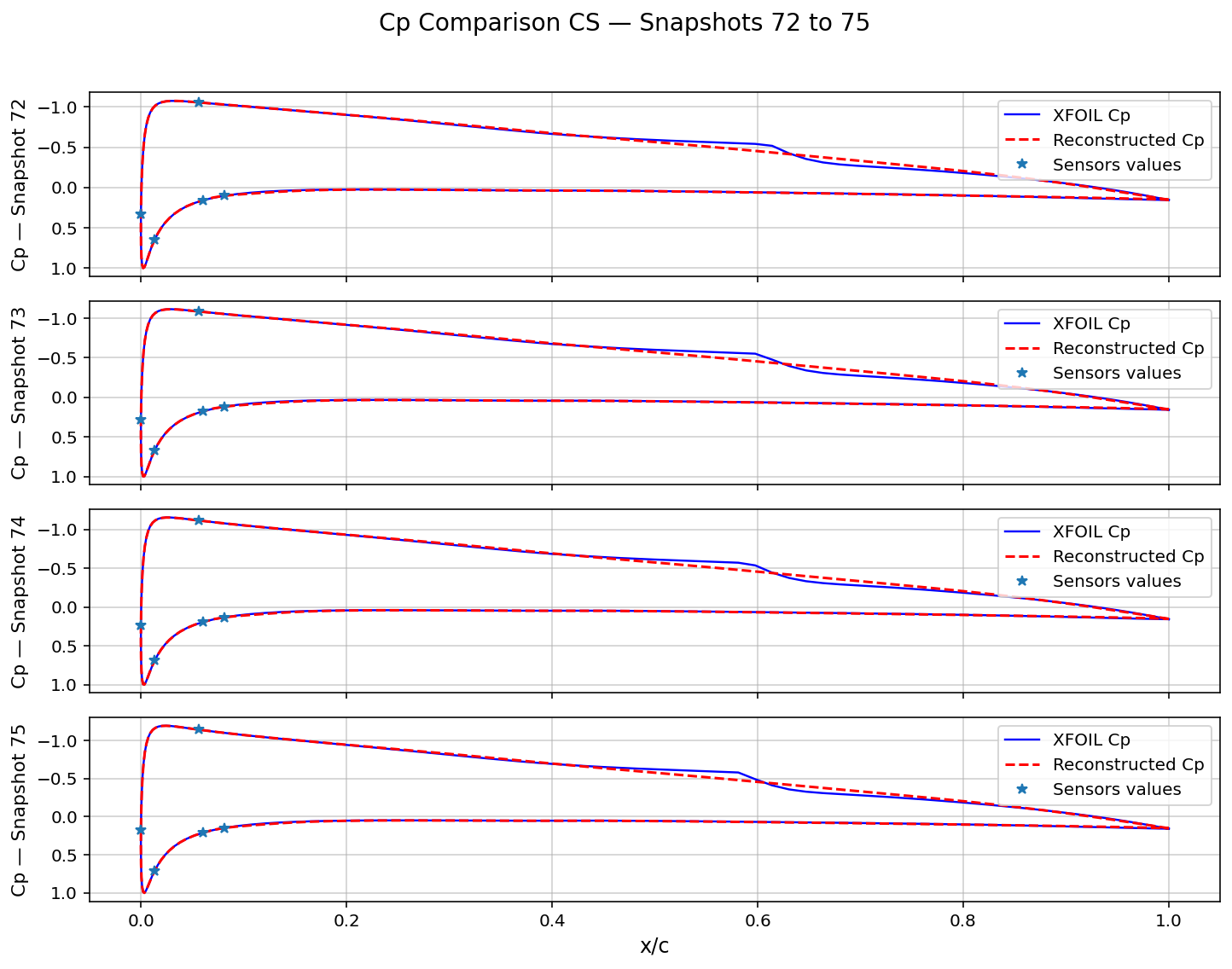}
    \end{minipage}
    
    \vspace{0.7cm}
    
    \begin{minipage}[t]{0.48\textwidth}
        \centering
        \includegraphics[width=\textwidth, height=0.34\textheight, keepaspectratio]{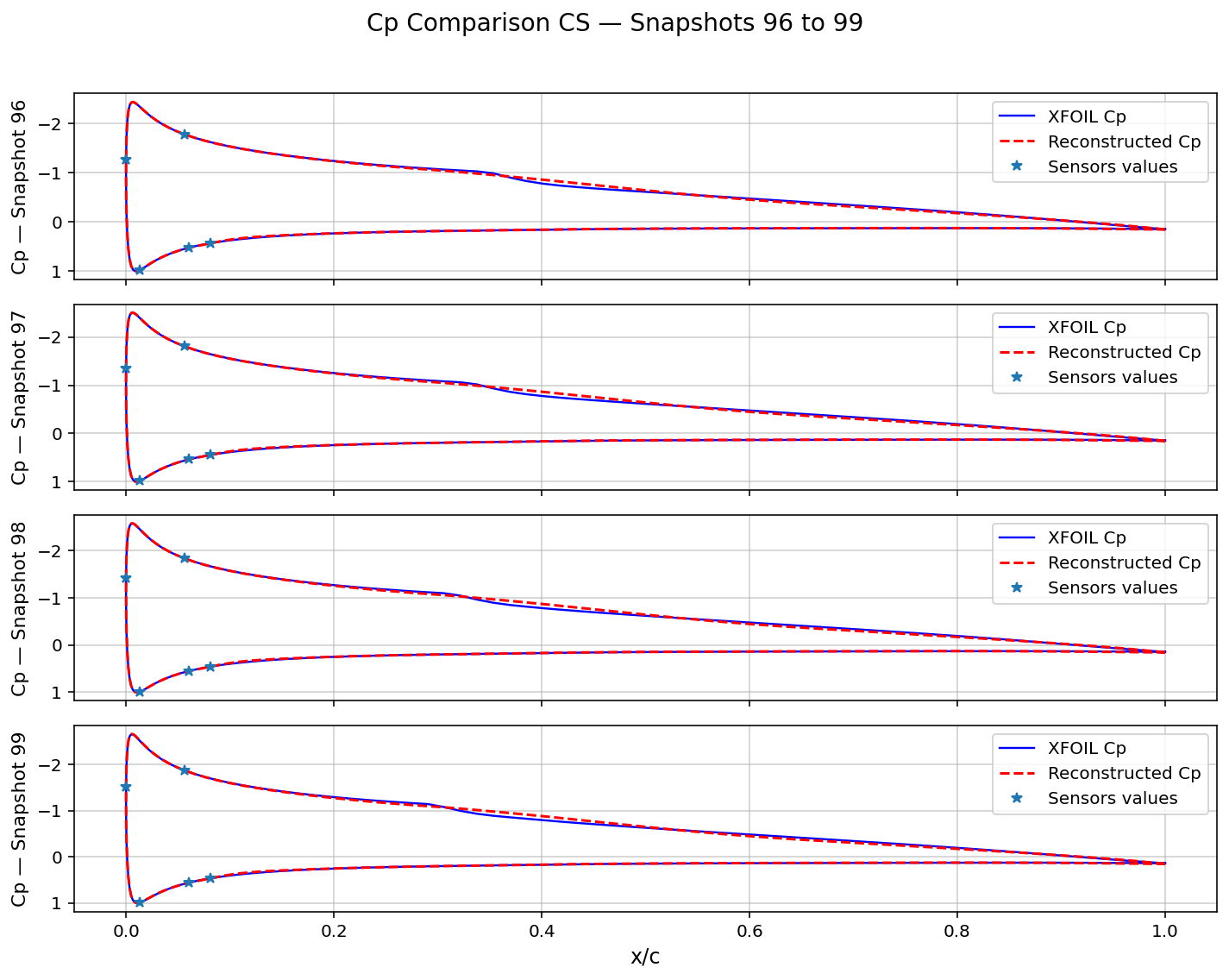}
    \end{minipage}
    
    \caption{Comparison of pressure coefficient \(C_p\) distributions for the NACA 2412 test set: XFOIL reference (blue solid) versus reconstruction obtained with the Inverse Cutting Sphere algorithm (Algorithm \ref{alg: inv inexact}) at \(\varepsilon = 0.095\) (red dashed). Blue stars mark the sensor measurements. Snapshots correspond to angles of attack from \(-7^\circ\) (snapshot 0) to \(+7^\circ\) (snapshot 99).}
    \label{fig:cp-comparison-cs}
\end{figure*}
	
	\section{Conclusion}

In this work we addressed the sensor placement problem for signal reconstruction in the absence of an explicit dynamical model. We formulated the problem as a nonconvex combinatorial optimisation task and showed that it can be recast as a weakly convex constrained projection problem. The reformulations proposed enabled us to apply the Inexact Cutting Sphere algorithm, thereby obtaining, for the first time, \emph{$\varepsilon$-global} solutions for the sensor selection problem.

We further proposed the Inverse Cutting Sphere algorithm, a novel method that starts from any feasible solution and either improves it by a prescribed amount $\varepsilon$ or certifies its $\varepsilon$-global optimality. The algorithm is particularly useful in practice because it can be warm-started from existing heuristics such as QDEIM and provides a rigorous optimality certificate when it terminates with $\texttt{STOP}=\texttt{TRUE}$.

Extensive numerical experiments were conducted on pressure reconstruction tasks for four NACA airfoils using high-fidelity XFOIL data. The results demonstrate that the cutting-sphere-based methods consistently achieve better or equal reconstruction accuracy compared with the widely used QDEIM heuristic, while also yielding superior values of the underlying optimality criteria ($-\log\det$ or condition number). At the same time, the experiments confirmed that QDEIM already produces high-quality sensor placements.

The main practical limitation of the proposed approach remains its computational cost, which grows with the number of constraints generated by the outer-approximation scheme. Nevertheless, the framework offers a valuable tool for sensor placement design when solution quality is more important than fast performance, and provides a rigorous benchmark against which faster heuristics can be evaluated.

Future research directions include the extension of the methodology to dynamic and time-varying sensor placement problems and the development of accelerated variants of the cutting-sphere algorithms in order to solve other scientific applications. We also want to extend our approach to sensors placement for classification, see \cite{BruntonClass}, and to the non-deterministic case, see \cite{kakasenko2026bridging}.


		\bibliographystyle{plain}
		\bibliography{Sreference.bib}

\end{document}